\documentclass[11pt]{amsart}
\usepackage{setspace}
\usepackage{amssymb, enumitem}
\usepackage[all]{xy}
\usepackage{hyperref, aliascnt, graphicx}
\usepackage{amscd}

\newtheorem{lma}{Lemma}[section]

\newaliascnt{thmCt}{lma}
\newtheorem{thm}[thmCt]{Theorem}
\aliascntresetthe{thmCt}

\newaliascnt{corCt}{lma}
\newtheorem{cor}[corCt]{Corollary}
\aliascntresetthe{corCt}

\newaliascnt{prpCt}{lma}
\newtheorem{prp}[prpCt]{Proposition}
\aliascntresetthe{prpCt}

\newtheorem*{thm*}{Theorem}
\newtheorem*{cor*}{Corollary}
\newtheorem*{prop*}{Proposition}

\theoremstyle{definition}

\newaliascnt{pgrCt}{lma}
\newtheorem{pgr}[pgrCt]{}
\aliascntresetthe{pgrCt}

\newaliascnt{dfnCt}{lma}
\newtheorem{dfn}[dfnCt]{Definition}
\aliascntresetthe{dfnCt}

\newaliascnt{rmkCt}{lma}
\newtheorem{rmk}[rmkCt]{Remark}
\aliascntresetthe{rmkCt}

\newaliascnt{rmksCt}{lma}

\aliascntresetthe{rmksCt}

\newaliascnt{exaCt}{lma}

\aliascntresetthe{exaCt}

\newaliascnt{egCt}{lma}
\newtheorem{eg}[egCt]{Example}
\aliascntresetthe{egCt}

\newaliascnt{qstCt}{lma}

\aliascntresetthe{qstCt}

\newaliascnt{pbmCt}{lma}

\aliascntresetthe{pbmCt}

\newaliascnt{ntnCt}{lma}

\aliascntresetthe{ntnCt}

\numberwithin{equation}{section}

\author{Marat Aukhadiev}
\title[C*-theory of the universal inverse semigroups]{An inverse semigroup approach to the C*-algebras and crossed products of cancellative semigroups }

\address{WWU M{\"u}nster, Einsteinstr. 62, 48149 M{\"u}nster, Germany.\newline Kazan State Power Engineering University, Krasnoselskaya 51, 420066 Kazan, Russia.\newline m.aukhadiev@uni-muenster.de}

\begin{document}

\begin{abstract}
	We give a new definition of the semigroup C*-algebra of a left cancellative semigroup, which resolves  problems of the construction by X.~Li. Namely, the new construction is functorial,   and the independence of ideals in the semigroup does not influence  the independence of the generators. It has  a group C*-algebra as a natural quotient. The C*-algebra of the old construction is a quotient of the new one. All this  applies both to the full and reduced C*-algebras. 	 The construction is based on the universal inverse semigroup generated by a left cancellative semigroup. We apply this approach to connect amenability of a semigroup to nuclearity of its C*-algebra. Large classes of actions of these semigroups are in one-to-one correspondence, and the crossed products are isomorphic.  A crossed product of a left Ore semigroup is isomorphic to the partial crossed product of the generated group. 
\end{abstract}

\maketitle

\section{Introduction}

Working with  semigroup C*-algebras of cancellative semigroups we face several significant problems.  The full semigroup C*-algebra $C^*(S)$ of a left cancellative semigroup $S$ was introduced by X. Li in \cite{Li}, and it was immediately noticed by the author in Section 2.5 that  this construction is not functorial, i.e. not every semigroup morphism of two cancellative semigroups extends to a *-homomorphism of their C*-algebras. It fails already for a morphism with domain the free monoid, as we show in \autoref{eg:free}. The reason for such behavior is that the generators of $C^*(S)$ imitate the left regular representation of $S$, while for functoriality one needs to consider a larger class of isometric representations. 

Another problem concerns the reduced semigroup C*-algebra. Many results in \cite{Li}, \cite{Li2} and in other papers on the subject assume independence of the constructible right ideals in $S$.  But this fails even in the simplest example of an abelian semigroup $\mathbb{Z}_+\setminus\{1\}$ with usual addition operation, see \autoref{pgr:independence}.

We solve these problems by constructing the universal inverse semigroup of $S$, and associating to $S$ the full and reduced C*-algebras of this inverse semigroup. These algebras have quotients isomorphic to $C^*(S)$ and $C^*_r(S)$ correspondingly, and a certain quotient is a group C*-algebra for some group associated with $S$.

An inverse semigroup naturally arises in the left regular representation of a left cancellative semigroup $S$, and relations between the C*-algebras of these two semigroups were studied by several  authors. The latest results in this direction were obtained by M.~Norling in \cite{Norling}, where $C^*_r(S)$ is described as a quotient of the C*-algebra of the left inverse hull $I_l(S)$ of $S$. The author also gives a surjective homomorphism between the full C*-algebras $C^*(S)\to C^*(I_l(S))$.  One can see that the full semigroup C*-algebra $C^*(S)$ is by definition a C*-algebra of an inverse semigroup $W$ generated by the elements of $S$ as isometries, where idempotents correspond to the constructible right ideals in $S$. These ideals are domains and images of operators in $I_l(S)$.

The semigroup $I_l(S)$ represents only the canonical action of $S$ on itself, and cannot capture all possible actions of $S$, neither can $W$.
As we show, a more efficient way is first to embed $S$ in a universal inverse semigroup, and then obtain through it the C*-algebra of $S$, and actions and crossed products by $S$. 
For this purpose we use the notion of a free inverse semigroup and the results on the problem of embedding cancellative semigroups in inverse semigroups.

Any injective action of a cancellative semigroup generates an inverse semigroup of partial bijections. This leads to a construction of the universal inverse semigroup $S^*$ generated by $S$, as we explain in Section \ref{s1}. We give a description of $S^*$ and its relation to the left inverse hull. The existence follows from the work \cite{Schein} of B.~Shain.

In Section 3 we answer the question when the universal inverse semigroup $S^*$ is E-unitary. This question is important, because so far the class of E-unitary inverse semigroups is the most well-studied. The answer is that $S^*$ is E-unitary precisely when $S$ is embeddable in a group. In this case we give a concrete model for $S^*$, describing it as a semigroup generated by an action on the group $G$ generated by $S$. This action is the one which gives the isomorphism $C^*(S^*)\cong C^*(E)\rtimes G$ due to a result by Milan and Steinberg in \cite{Milan}.

Section 4 contains a study of the reduced semigroup C*-algebra of $S^*$. We prove that the left regular representation of $S$ is a subrepresentation of the left regular representation $V$ of $S^*$. It follows that $C^*_r(S)$ is a quotient of $C^*_r(S^*)$. When $S$ embeds in a group, the model for $S^*$ described in Section 3 allows to compute $V$. We prove that $V$ is a direct sum of representations, each of them can be realized  as some restriction of the left regular representation of $G$ onto a subspace of $\ell^2(G)$, where $G$ is a group generated by $S$. And the left regular representation of $S$ is one of these summands.

We compare the full semigroup C*-algebras of $S$ and $S^*$ in Section 5.  By virtue of general theory of inverse semigroups, a natural quotient of $C^*(S^*)$ is a group C*-algebra $C^*(G)$, where G is the maximal group homomorphic image of $S^*$, and the same holds for their reduced versions. In the case when $S$ embeds in a group, $G$ is isomorphic to the group generated by $S$. We point out that the assignment $S\mapsto C^*(S^*)$ is functorial, unlike the assignment $S\mapsto C^*(S)$. \autoref{eg:free} of a free monoid $\mathbb{F}_n^+$ illustrates this difference. Unlike $C^*(\mathbb{F}_n^+)$, which almost never admits a homomorphism onto $C^*(S)$ for an $n$-generated left cancellative semigroup $S$, the quotient map $C^*({\mathbb{F}_n^+}^*)\mapsto C^*(S^*)$ always exists. A natural quotient of $C^*({\mathbb{F}_n^+}^*)$ is the Cuntz algebra $\mathcal{O}_n$.

In \autoref{pgr:independence} we note that the idempotent generators of $S^*$ under the left regular representation are linearly independent, due to the general theory of inverse semigroups. Therefore, an important question of independence of constructible right ideals has no importance for $C^*_r(S^*)$. We note that $C^*(S^*)$ and $C^*_r(S^*)$ are nuclear for an abelian semigroup $S$.

We apply our constructions to the connection between amenability of a cancellative semigroup $S$ and nuclearity of its reduced C*-algebra in Section \ref{secamenab}. In particular, we prove that if $S$ embeds in an amenable group, then $C^*_r(S)$ is nuclear. Thus we generalize results of \cite{Li2} on this question. Moreover, in the above mentioned case, $S^*$ has the weak containment property and $C^*_r(S^*)$ is nuclear. We also show that amenability of $S$ implies amenability of $S^*$ in any case. 

It is a known fact that $C^*_r({\mathbb{F}_n^+})$ is nuclear despite the fact that $\mathbb{F}_n^+$ is not amenable. Using our constructions, we obtain that  $C^*_r({\mathbb{F}_n^+}^*)$ is not nuclear, which makes it a more natural candidate for the C*-algebra associated with the free monoid. 

In Section \ref{seccross} we give connections between actions of $S$ and $S^*$ on spaces, C*-algebras, and prove isomorphisms of crossed products. We consider actions of cancellative semigroups by endomorphisms; the case of automorphisms was studied in \cite{Li}. According to the definition by \cite{Sieben}, an  inverse semigroup acts on a C*-algebra by *-isomorphisms between closed two-sided *-ideals of the C*-algebra. Therefore, with a view to connect actions by cancellative semigroups with actions by inverse semigroups, we are forced to restrict ourselves to the case when the images of endomorphisms  are ideals.  

First we prove that injective actions of a cancellative semigroup $S$  are in one-to-one correspondence with unital actions of $S^*$, so that an action of one induces an action of the other. Then we prove an isomorphism between the crossed products $A\rtimes_{\alpha} S$ and $A\rtimes_{\tilde{\alpha}} S^*$, where $\tilde{\alpha}$ is induced by $\alpha$ or the other way round. The case of a unital C*-algebra $A$ is more common for the crossed products by cancellative semigroups, hence we consider the unital and the non-unital case separately. We use the definition of a crossed product with non-unital C*-algebra of N.~Larsen \cite{Larsen}. This result allows us to describe $C^*(S^*)$ as a crossed product of a commutative C*-algebra by $S$, and $C^*(S)$ as a crossed product by $S^*$. If $S$ is a left Ore semigroup, then any unital action of $S^*$ can be dilated to an action of a group, and the crossed product is Morita equivalent to the group crossed product.

In Section \ref{secore} we study connections of actions and crossed products of semigroups with partial actions and partial crossed products of groups. If $S$ is embeddable in a group, the model for $S^*$ gives a partial action of the group $G$ (the group generated by $S$), such that $C^*(S^*)$ is isomorphic to a partial crossed product by $G$. This isomorphism is precisely the one given by Milan and Steinberg in \cite{Milan}. A stronger result holds in the case of a left Ore semigroup, that is, a semigroup $S$ such that $G=S^{-1}S$ is a group. Using the semigroup $S^*$ and the work of Exel and Vieira \cite{Exel2}, we prove that any injective action of $S$ generates a partial action of $G$, and the corresponding crossed products are isomorphic.

\begin{pgr}\label{pgr:invsem} Let us recall the main definitions and facts used in this paper. Let $P$ be a semigroup. Elements $x$ and $x^*$ in $P$ are called \emph{inverse} to each other if $$xx^*x=x,\mbox{ and } x^*xx^*=x^*.$$ The semigroup $P$ is called an \emph{inverse semigroup} if for any $x\in P$ there exists a \textbf{unique}  inverse element $x^*\in P$. Further, $P$ always stands for an inverse semigroup. 
We proceed to recall basic facts on inverse semigroups.

\begin{thm}\label{vagner}(V.~V.~Vagner \cite{Vagner}). For a semigroup $S$ in which every element has an inverse, uniqueness of inverses is equivalent to the requirement that all idempotents in $S$ commute.
\end{thm}

The set of idempotents of an inverse semigroup $P$ forms a commutative semigroup denoted $E(P)$. In fact, $$E(P)=\{ xx^*| x\in P\}=\{x^*x| x\in P \}.$$ 

Every inverse semigroup $P$ admits a universal morphism onto a group $G(P)$, which is the quotient by the congruence: $s\sim t$ if $se=te$ for some $e\in E(P)$. The group $G(P)$ is called \emph{the maximal group homomorphic image} of $P$. Note that $G(P)$ is always trivial if $P$ contains a zero, i.e. an element denoted $0$, satisfying $0\cdot a=a\cdot 0=0$ for any $a\in P$. Let $\sigma\colon P\to G(P)$ denote the quotient homomorphism onto the maximal group homomorphic image of $P$. The semigroup $P$ is called \emph{E-unitary} if $\sigma^{-1}(1)=E(P)$.
\end{pgr}

\begin{pgr} A semigroup $S$ is called \emph{left (right) cancellative} if for any $a,b,c\in S$  the equation $ab=ac$ ($ba=ca$) implies $b=c$. A unit in a semigroup is an element denoted by $1$, satisfying $a\cdot 1=1\cdot a=a$ for any $a$ in the semigroup.   %Further in this section by saying ``cancellative'' we mean ``left cancellative'' and $S$ always stands for such semigroup.
\end{pgr}

\begin{pgr} Let us compare inverse and left cancellative semigroups. These two classes of  semigroups have radical differences, which follow directly from the definitions. The notion of an inverse semigroup is a natural generalization of the notion of a group, where a group inverse element ($ss^{-1}=1$ and $s^{-1}s=1$) is substituted by a ``generalized inverse'' ($ss^*s=s$ and $s^*ss^*=s^*$).  This is the reason why at the early stage the inverse semigroups were called ``generalized groups'' (\cite{Vagner}). Inverse semigroups have many idempotents and may have a zero, while a left cancellative semigroup may have only one idempotent, namely the unit element, and no zero element. A left cancellative semigroup is very often embedded in a group, while an inverse semigroup is a subsemigroup in a group only if it is a group itself. In particular, the intersection of these classes is the class of groups. 

One meets the consequence of these differences in the theory of semigroup C*-algebras, starting with the left (right) regular representation. An inverse semigroup is represented on itself by partial bijections. A left cancellative semigroup is represented on itself by injective maps, where the domain is the whole semigroup.  An inverse semigroup has an involution, which is a map assigning to every element of $S$ its inverse element. And the presence of an involution makes it very natural to consider *-representations in $B(H)$. Despite the different nature, soon after the establishment of inverse semigroups, it was noticed that these two classes are closely related. In the following section we construct a universal inverse semigroup generated by a left cancellative semigroup.
\end{pgr}

\section{Universal inverse semigroup}\label{s1}

\begin{pgr} Recall the basic example of an inverse semigroup. Let $X$ be a set, and let $Y\subset X$. A one-to-one map $\alpha\colon Y\to  X$ is called a \emph{partial bijection} of $X$. In particular, any injective map $X\to X$ is a partial bijection of $X$. Suppose that $\alpha$ and $\beta$ are partial bijections of $X$ with domains $Y$ and $ Z$ respectively. Then the product $\alpha\beta$ is defined to be a composition of $\alpha$ and $\beta$ with domain $\beta^{-1}(\beta(Z)\cap Y)$. The set $\mathcal{I}(X)$ of partial bijections with this product forms an inverse semigroup called \emph{the symmetric inverse semigroup of} $X$. Note that this semigroup contains a zero and a unit.
\end{pgr}

In what follows we always assume that every semigroup contains a unit element, denoted by 1.% It is known that  a unit can be always added to an inverse semigroup if it is missing, without changing the C*-algebra theory of this semigroup.

\begin{pgr} The first inverse semigroup associated with a left cancellative semigroup $S$ arises from the left regular representation of $S$. 
For any $a\in S$, define an operator of left ranslation $\lambda_a\colon S\to S$ by $\lambda_a(b)=ab$ for all $b\in S$. Since $S$ is left cancellative, each $\lambda_a$ is injective.  Then  $\{\lambda_a| a\in S\}$ forms a semigroup of partial bijections on $S$, and it is a subsemigroup of $\mathcal{I}(S)$. The inverse of $\lambda_a$ is a partial bijection with a domain $\{ab|b\in S\}$. Taking inverses of all such partial bijections and products of them, one obtains a subsemigroup of $\mathcal{I}(S)$. This is an inverse semigroup called \emph{the left inverse hull of} $S$ (\cite{Che}), denoted $I_l(S)$.
\end{pgr}

\begin{pgr} More generally, suppose we have an injective action $\alpha$ of a left cancellative semigroup $S$ on a space $X$. It means that for every $s\in S$, the map $\alpha_s\colon X\to X$ is injective and $\alpha_s\circ \alpha_t=\alpha_{st}$ for all $t\in S$.  Denote the image of $\alpha_s$  by $D_s\subset X$. Then $\alpha_s$  is a bijection between $X$ and its image $D_s$, and there exists an inverse map, which we denote by $\alpha_s^*\colon D_s\to X$. For convenience set $D_{s^*}=X$ for every $s\in S$. One can easily verify that $D_{st}=\alpha_s(D_t)$.  

Clearly, $\alpha_s^*\circ \alpha_s$ is the identity on $X$ and $\alpha_s\circ \alpha_s^*$ is the identity on $D_s$. It follows that $\alpha_s\circ\alpha_s^*\circ \alpha_s=\alpha_s$ and $\alpha_s^*\circ\alpha_s\circ \alpha_s^*=\alpha_s^*$.  But the composition $\alpha_s^*\circ\alpha_t$ is defined only on a subset of $X$, namely on $\alpha_t^*(D_s\cap D_t)$. Thus we put $D_{(s^*t)^*}=\alpha_t^*(D_s\cap D_t)$ and define the product $\alpha_s^*\alpha_t=(\alpha_s^*\circ\alpha_t)|_{ D_{t^*s}}$. One should check that this definition is compatible with multiplication in $S$, namely $\alpha_{st}^*\alpha_v=\alpha_t^*\alpha_s^*\alpha_v$. Continuing this way we define all finite products $w$ of the maps from the collection $F=\{\alpha_s, \alpha_t^* \mbox{ for all } t,s\in S\}$ with domain $D_w$.  We put $\alpha_s^{**}=\alpha_s$ for all $s\in S$. 

It is easy to see that for $a_1,...,\ a_n\in F$, the element $(a_1a_2...a_n)^*=a_n^*...a_2^*a_1^*$ is the inverse (in a semigroup sense) of $w=a_1a_2...a_n$, and that $ww^*$ and $w^*w$ are idempotents. Obviously, $w^*w$ and $v^*v$ commute for any words $v,w$ and any idempotent has the form $w^*w$. Thus, we get an inverse semigroup, which is a subsemigroup in a set of all partial bijections on a space $X$. Note that in the case that $\alpha$ is an action of $S$ by injective maps, we have $\alpha_s^*\alpha_s=\mathrm{id}$, so $\alpha_s$ is an isometry. We have verified the following statement.
\end{pgr}

\begin{lma}\label{act} An injective action $\alpha$ of a left cancellative semigroup $S$ on a space $X$ generates an inverse semigroup $S^*_\alpha\subset \mathcal{I}(X)$.
\end{lma}

This motivates a notion of a universal inverse semigroup generated by a left (right) cancellative semigroup. A problem of embedding  a semigroup in an inverse semigroup is analogous to the well-known and widely studied problem of embedding  it in a group. Recall a famous result of O.~Ore and P.~Dubreil on this problem, which we will use later.

\begin{thm}\label{ore} (\cite{Dubreil}) A semigroup $S$ can be embedded into a group $G$ such that $G=S^{-1}S$ if and only if it is left and right cancellative and for any $p,q\in S$ we have $Sp\cap Sq\neq\emptyset$. The group $G$ is called the group of left quotients of $S$.\end{thm}

\begin{pgr} The question of embedding a semigroup in an inverse semigroup is more general and is approached in the following way. For any set $X$ there exists a free inverse semigroup $F(X)$, generated by $X$ (see \cite{Preston} for the proof). Thus, if $S$ is a semigroup, we can consider the quotient of $F(S)$ by all the relations in  $S$. Namely, if $xy=z$ in $S$ we put $xy \sim z$ in $F(S)$, and the same for inverses. The resulting semigroup is called the free inverse semigroup of $S$, and we denote it $S^F$. So, the question becomes whether the natural map $S\to S^F$ is an embedding. In fact, there were found many sufficient conditions for this to hold. Among these results we mention the most important for our research.
\end{pgr}

\begin{thm}\label{Schein} (B. Shain \cite{Schein}). 
If a semigroup $S$ is left (or right) cancellative, then $S$ can be embedded into an inverse semigroup.\end{thm}

\begin{pgr} Therefore, working with a left cancellative semigroup $S$ we always have an embedding $S \hookrightarrow S^F$. B. Shain also gave an explicit form of $S^F$. Namely, it is the semigroup generated by the set $$\{v_p,v_p^*\colon  p\in S,\ v_pv_p^*v_p=v_p,\ v_p^*v_pv_p^*=v_p^* \}$$ with the additional requirement that all idempotents in $S^F$ commute. Since we want $S$ to be represented by isometries, we take the congruence on  $S^F$  generated by the equivalence relation $v_p^*v_p\sim 1$ for all $p\in S$. The quotient inverse semigroup, denoted by $S^*$, is then generated by isometries $v_p$ for $p\in S$. The semigroup $S^*$ is in some sense the largest inverse semigroup generated by $S$ as a semigroup of isometries. 
\end{pgr}

\begin{pgr}\label{pgr:hull} For the moment we have mentioned two inverse semigroups associated with a given left cancellative semigroup $S$, constructed in different ways: the universal inverse semigroup $S^*$, and the left inverse hull $I_l(S)$. In order to see the relation between them, we make a short review of \cite{Che} and the notion of the left inverse hull. Note that we adapt all the relations and notations, because, unlike the left cancellative case selected here, the semigroup $S$  in \cite{Che} is right cancellative. Also, a semigroup with a unit is always left and right reductive, so the semigroup $S$ which we consider fits into the requirements of \cite{Che}. For the generators of $F(S)$, we use symbols $v_p$ for $p\in S$ as above.

 The left inverse hull $I_l(S)$ is proved  in \cite{Che} to be isomorphic to the quotient of $F(S)$ by four collections of relations. In the notation of \cite{Che}, the above mentioned relations $v_p^*v_p\sim 1$ for all $p\in S$ are denoted $R_1$, and $R_4$ is a collection of all relations in $S$. Therefore, passing to the quotient of $F(S)$ by the congruence generated by $R_1\cup R_4$ we obtain exactly $S^*$ (introduced above). The relations $R_2$  and $R_3$ ensure that the elements are equivalent if the corresponding operators in $I_l(S)$ have the same domain and act in the same way (see below). 

An explicit form for the domains of the maps in $I_l(S)$ is given using the notation in \cite{Li}. For any subset $A\subset S$ and any $a\in S$ set
\begin{equation}\label{aa} aA=\{ax|\ x\in A\} \end{equation}
\begin{equation}\label{aca} a^{-1}A=\{x\in S| xa\in A\}\end{equation}
In \cite{Che} these sets are denoted $aA$ and $a: A$ respectively. Then the domain of $\phi=\lambda_{a_1}^{-1}\lambda_{a_2}...\lambda_{a_{n-1}}^{-1}\lambda_{a_n}\in I_l(S)$ is the set $a_{n}^{-1}(a_{n-1}... a_{2}^{-1}(a_1 S))$, which is the image of $\lambda_{a_n}^{-1}\lambda_{a_{n-1}}...\lambda_{a_{2}}^{-1}\lambda_{a_1}=\phi^*\in I_l(S)$. Such domains are right ideals in $S$. In fact, these sets were called \emph{the constructible right ideals} of $S$ by X. Li in \cite{Li} and used there for the definition of $C^*(S)$ and for the study of $C^*_r(S)$. The set of such ideals with an empty set is denoted by $\mathcal{J}$, so $\mathcal{J}$ is the set of all domains (and images) of all maps in $I_l(S)$.

We can now formulate the relations $R_2$ and $R_3$ on $F(S)$. The first of them introduces the zero element.
 $$v_{a_1}^*v_{a_2}...v_{a_{n-1}}^*v_{a_n}\sim_{R_2} uv_{a_1}^*v_{a_2}...v_{a_{n-1}}^*v_{a_n}w \mbox{ for any } u,w\in F(S) $$
$$\mbox{ iff }  a_{n}^{-1}(a_{n-1}... a_{2}^{-1}(a_1 S))=\emptyset; $$

The relations $R_3$ ask for elements to be equivalent if when represented on $S$ they act in the same way.
$$v_{a_1}^*v_{a_2}...v_{a_{n-1}}^*v_{a_n}\sim_{R_3} v_{b_1}^*v_{b_2}...v_{b_{k-1}}^*v_{b_k} $$
$$\mbox{ iff }  a_{n}^{-1}(a_{n-1}... a_{2}^{-1}(a_1 S))= b_{k}^{-1}(b_{k-1}... b_{2}^{-1}(b_1 S))$$
$$ \mbox{ and for any } x\in a_{n}^{-1}(a_{n-1}... a_{2}^{-1}(a_1 S)) $$
$$\mbox{ there exist } x_1,...x_n,y_1,...y_k\in S, \mbox{ such that }x_1=y_1 \mbox{ and }$$
$$  a_nx=a_{n-1}x_n,\ a_{n-2}x_n=a_{n-3}x_{n-1},\ ...,\ a_2x_2=a_1x_1,   $$
$$ b_kx=b_{k-1}y_k,\ b_{k-2}y_k=b_{k-3}y_{k-1},\ ...,\ b_2y_2=b_1y_1. $$
\end{pgr}

The main result of \cite{Che} is that  $I_l(S)$ is isomorphic to the quotient $F(S)/R^*$, where $R^*$ is the congruence generated by $R=R_1\cup R_2\cup R_3\cup R_4$. Comparing this fact with our definition of  $S^*$, we obtain the following.

\begin{lma}\label{lma:allfree} There are surjective homomorphisms of inverse semigroups
\begin{equation}\label{fil} F(S)\stackrel{\alpha}{\rightarrow}S^F\stackrel{\beta}{\rightarrow} S^*\stackrel{\gamma}{\rightarrow} I_l(S), \end{equation}
where $\alpha$ is the quotient by $R_4$, $\beta$ by $R_1$, and $\gamma$ by $R_2\cup R_3$. \end{lma}

\section{E-unitarity and a model for the universal inverse semigroup}

As explained in \autoref{pgr:invsem}, for any inverse semigroup $P$ there exists a maximal group homomorphic image $G(P)$, which is not trivial if the semigroup does not contain a zero. This is the case for $S^*$ due to its definition. 

\begin{lma}\label{lma:eunitary}
For a left cancellative semigroup $S$, its universal inverse semigroup $S^*$ (as well as $S^F$) is E-unitary  if $S$ is embeddable in a group. In this case $G(S^*)=G(S^F)$ is isomorphic to the group generated by $S$.
\end{lma}
\begin{proof}
Firstly,  $S^F\to G(S^F)$ factors through $\beta \colon S^F\to S^*$ given by $v_a^*v_a\sim 1$. Therefore, by maximality $S^F$ and $S^*$ share the same maximal group homomorphic image, i.e. $G(S^*)=G(S^F)$.

Let $S^*$ be E-unitary and assume $\sigma(v_a)=\sigma(v_b)$ in $G(S^*)$. Then $v_bv_a^*$ is a self-adjoint idempotent, and we obtain
$$v_a=v_av_b^*v_b=v_bv_a^*v_b=v_bv_a^*v_bv_a^*v_a=v_bv_a^*v_a=v_b$$
Therefore, $\sigma$ is an embedding of the semigroup $\{v_a\colon a\in S\}$ in $G(S^*)$, which can be identified with $S$. Since $v_a$ are generators of $S^*$, their image under $\sigma$ generates the group $G(S^*)$ by maximality. 

If $S^F$ is E-unitary and $\sigma\beta(v_a)=\sigma\beta(v_b)$, then $v_bv_a^*$ is a self-adjoint idempotent in $S^F$ and therefore so is $v_bv_a^*\in S^*$. Hence, $S$ is embeddable in a group in this case as well.
\end{proof}

\begin{pgr}\label{pgr:model}
In the case when $S$ is embeddable in a group, there exists a model realizing the universal inverse semigroups $S^*$ and $S^F$. First, let us give a model for $S^F$, the predecessor of $S^*$. 

For any set $X$ there exists a free inverse semigroup $F(X)$. The proof in   \cite{Preston} and \cite{Banakh}, uses essentially the free group $G(X)$ on the set $X$, and embedding of $X$ in the free group. The idea of our model is based on the model given in \cite{Banakh} for $F(X)$.
 
As shown in \autoref{lma:allfree}, $S^F$ is a quotient of the free inverse semigroup on $F(S)$ by congruence generated by relations in the semigroup $S$. Therefore, the quotient of the model for $F(S)$ gives a correct model for $S^F$ if $S$ is embeddable in a group. And in this case, a group generated by $S$ is used instead of the free group on the set. We denote by $G$ a group generated by $S$, and let $exp(G)$ be the set of all finite subsets of $G$ containing $1$. Due to maximality of the maximal group homomorphic image of $S^F$, and the fact that $S^F$ is generated by elements $v_a$, we have $G(S^F)=G$. Denote by $\sigma$ the quotient map $S^F\to G$. 

 For any $A\in exp(G)$, $g\in G, a,b\in S$ define a relation
 \begin{equation}\label{sim} A\cup\{1,g,ga,gab\} \sim A\cup\{1,g,gab\}\end{equation}
 One can easily verify that this is an equivalence relation and generates a congruence on $exp(G)$. 
 It follows in particular that 
$$\{1,g,ga^{-1},ga^{-1}b^{-1}\}\sim \{1,g, ga^{-1}b^{-1}\},\quad \{1,g,ga^{-1},gb\}\sim \{1,ga^{-1}, gb\}.$$ 
We will formally write $g_1\leq g\leq g_2$ if $g=g_1a$, $g_2=gb$ for some $a,b\in S$. Then the relation $\sim$ on $exp(G)$ can be formulated as
$$A\cup\{g\}\sim A \mbox{ if and only if there exist } g_1,g_2\in A \mbox{ such that } g_1\leq g\leq g_2.$$

For any $A\in exp(G)$ denote its equivalence class by $[A]$ and the set of all equivalence classes by $exp(G)/\sim$. Clearly, $\sim$ is stable under taking union, i.e. 
$$A\sim A',\ B\sim B' \Rightarrow A\cup B\sim A'\cup B'.$$
 This equivalence is also stable under multiplication by elements of $G$ from the left:
$$gA=\{1\}\cup\{ga\colon a\in A\},$$
 so that the class $g[A]=[gA]$ is well-defined. Note that this is not an action of $G$ on $exp(G)$, because in general $(gh)A\neq g(hA)$.  And we have $g^{-1}(gA)=\{g^{-1}\}\cup A$ for any $A\in exp(G)$, $g\in G$.

Define a partial order on $exp(G)/\sim$, for any $A,B\in exp(G)$:
$$[A]\leq [B] \mbox{ if and only if there exists }  B'\sim B \mbox{ such that } A\subset B'$$
This means in particular, that for any $A,B\in exp(G)$ and $g\in G,a,b\in S$

$$A\subset B \Rightarrow [A]\leq  [B],$$
$$[\{1,g,ga\}]\leq [A\cup\{1,g, gab\}], \quad [\{1,ga^{-1},ga^{-1}b^{-1}\}]\leq [A\cup\{1,g, ga^{-1}b^{-1}\}],$$
$$[\{1,g,ga^{-1}\}]\leq  [A\cup\{1,gb, ga^{-1}\}],\quad [\{1,g,ga^{-1}\}]\leq  [A\cup\{1,g, ga^{-1}b^{-1}\}]  $$

The partial order $\leq$ on $exp(G)/\sim$ is stable under multiplication from the left by elements of $G$ due to stability of $\sim$.

We say that an element $g=a_1a_2...a_n\in G$ is written in a reduced form if $a_i$'s are alternating elements from $S$ and $S^{-1}$ with $a_i\neq a_{i+1}^{-1}$ for all $1\leq i\leq n-1$. Define a set $I_g=\{1,a_1,a_1a_2,...,a_1a_2...a_n\}\in exp(G)$ corresponding to a fixed reduced form of $g$. A set $A\in exp(G)$ is called \emph{full} if for any $g\in A$ it contains the subset $I_g$ for some reduced form of $g$. This means that any full set equals $I_{g_1}\cup I_{g_2}\cup...\cup I_{g_n}$ for some $g_i\in G$. Denote by $E$ the set of equivalence classes with respect to $\sim$ of full sets in $exp(G)$. Note that a full set may be equivalent to a non-full set.

Define the set
 $$\tilde{G}=\{(g,[A])\in G\times E \colon [I_g]\leq  [A] \mbox{ for some reduced form of }g\}.$$ 
 In fact, it is sufficient to require that $g\in A'$ for some full set $A'\sim A$. We define product and inverse operation on $\tilde{G}$:
\begin{equation}\label{modelprod}
(g,[A])(h,[B])=(gh,[A\cup gB]),
\end{equation} 
\begin{equation}\label{modelinv}
(g,[A])^*=(g^{-1},g^{-1}[A]).
\end{equation} 
Since $[I_g]\leq  [A]$, we have that $[I_{g^{-1}}]=g^{-1}[I_g]\leq g^{-1}[A]$. Hence $(g,[A])^*\in \tilde{G}$.

\begin{thm}\label{thm:modelsf}  The set $\tilde{G}$ with product and inverse operation defined by (\ref{modelprod}) and (\ref{modelinv}) forms an inverse semigroup isomorphic to $S^F$.
\end{thm}
\begin{proof}
To see that the product is associative, take $g,h,f\in G$ and full sets $A,B,C$ in $exp(G)$ and compute
$$A\cup gB\cup(gh)C=A\cup (gB\cup\{g\})\cup (ghC)=A\cup g(B\cup hC).$$
For any $(g,[A])\in\tilde{G}$ we may assume $I_g\subset A$, which implies $g(g^{-1}A)=A$. Consequently,
$$(g,[A])(g^{-1},g^{-1}[A])(g,[A]) =(g,[A\cup g(g^{-1}A)\cup A])=(g,[A]).$$
Checking similarly the corresponding equation for $(g^{-1},g^{-1}[A])$, we obtain that (\ref{modelinv}) defines an inverse element for $(g,[A])$.

The idempotents in $\tilde{G}$ correspond to elements of $E$: $$(g,[A])(g,[A])^*=(1,[A])\leftrightarrow [A]\in E.$$
The product $[A][B]=[A\cup B]$ on $E$ is commutative. Moreover, $[A]\leq[B]$ if and only if idempotents $(1,[B])\leq (1,[A])$.

Thus, $\tilde{G}$ is an inverse semigroup.

As mentioned above, the semigroup $S^F$ is a quotient of the free inverse semigroup $F(S)$ on $S$ as a set by equivalence relation generated by relations among elements in $S$. Therefore it is sufficient to show that the model described is a quotient of the model for $F(S)$ by relations in $S$. Let $T$ be the free monoid generated by the set $S$. For $a,b\in S$ we denote their product in $T$ by $ab$ and their product in $S$ by $a\cdot b$. Now $T$ is a cancellative monoid and clearly $T^F=F(S)$ by definition, with a maximal group homomorphic image $G(T)$ equal to the free group on the set $S$ and to $G(F(S))$. One can easily see that the model for $T^F$ coincides with the model in \cite{Preston} for $F(S)$, which we describe further.
   
	 Since $G(T)$ is a free group, every its element has a reduced form. For any $g\in G(T)$ with a reduced form $r(g)=a_1a_2...a_n$, where $a_i\in S\cup S^{-1}\subset G(T)$, we define
$$\hat{g}=\{1,a_1,a_1a_2,...,a_1a_2...a_n\}  $$
 The set $A\in exp(G(T))$ is called saturated if $g\in A$ implies $\hat{g}\subset A$. Then
$$F(S)=\{(a,A)\colon A \mbox{ is saturated and } r(a)\in A\}\subset G(T)\times exp(G(T))$$ 
The product and inverse operation on $F(S)$ are given by
$$(a,A)(b,B)=(ab,A\cup bB),$$
$$(a,A)^*=(a^{-1},a^{-1}A).$$
The idempotent semilattice in $F(S)$ consists of elements $(1,A)$ for all saturated sets $A$.

The generators of $F(S)$ are
	$$t_a=(a,\{1,a\}), t_a^*=(a^{-1},\{1,a^{-1}\})$$
for all $a\in S$, where $a^{-1}$ is the inverse of $a$ in $G(T)$. Note that for $a,b\in S$ we have
$$t_at_b=(ab, \{1,a,ab\}) \mbox{ and } t_{(a\cdot b)}=(a\cdot b,\{1,a\cdot b\}),$$
and the same for their inverses. The homomorphism $F(S)\to S^F$ is a quotient map by equivalence $t_at_b\sim t_{(a\cdot b)}$. This consists of equivalences on $G(T)$ and $exp(G(T))$:
$$ab\sim a\cdot b,  \{1,a,ab\}\sim \{1,a\cdot b\}.$$ 
It is easy to see that  the quotient of $G(T)$ by the first equivalence equals $G$. The second reduces subsets of $G(T)$ to subsets of $G$ and induces a new equivalence on it, which is given by (\ref{sim}). Under this equivalence  a  saturated subset of $G(T)$ turns into a full subset of $G$. Thus we obtain the model $\tilde{G}$.

 For the sake of completeness, we give an isomorphism between $\tilde{G}$ and $S^F$. For any element in $S^F$ we define a corresponding element in $G\times E$. Obviously, any element of $S^F$ can be  written as a product of alternating symbols of the type $v_a$ and $v_b^*$, for $a,b\in S$. We call such a form of an element \emph{a reduced form}. Let $s\in S^F$ be an element with a reduced form $v_{a_1}^*v_{a_2}...v_{a_n}^*$, where $a_1,..a_n\in S$. Define 
$$I_s=\{1,\sigma(v_{a_1}^*),\sigma(v_{a_1}^*v_{a_2}),...,\sigma(s)\}\in exp(G),$$
$$\tilde{s}=(\sigma(s),[I_s])\in G\times E.$$
Then because of the equivalence relation we put on $exp(G)$, the map $s\mapsto \tilde{s}$ does not depend on the choice of the reduced form for $s$. Clearly, this map is a  *-homomorphism.  

Let us define the reverse map $\tilde{G}\to S^F$. Consider a full set $$I_g=\{1,a_1^{-1},a_1^{-1}b_1,a_1^{-1}b_1a_2^{-1},...,a_1^{-1}b_1a_2^{-1}...a_n^{-1}b_n\}\in exp(G)$$
 corresponding to $g\in G$ and its reduced form $g=a_1^{-1}b_1a_2^{-1}...a_n^{-1}b_n$.
  
  First assume that $I_g$ contains precisely one element of $S^{-1}$ (or $S$), denote it $a_1^{-1}$ (or $a_1$). Then assume that  $I_g$ contains precisely one element in $a_1^{-1}S$ (or $a_1S^{-1}$), etc. Under all these assumptions we get a unique chain $a_1^{-1},b_1,...a_n^{-1},b_n$ and define
  $$s_g=v_{a_1}^*v_{b_1}...v_{a_n}^*v_{b_n},$$
 as an element corresponding to $ (g,[I_g])$.
 
  If on some step this is not true and we have, for instance $$a_1^{-1}b_1...a_k^{-1}b_ka_{k+1}^{-1}b_{k+1}...a_i^{-1}b_i=a_1^{-1}b_1...a_k^{-1}b_kc^{-1}$$ 
  for some $c\in S$, then we split $I_g$ into a union of  sets
$$I_{g_1}=\{1,a_1^{-1},... ,a_1^{-1}b_1...b_k,a_1^{-1}...b_kc^{-1},$$ $$a_1^{-1}...b_kc^{-1}a_{i+1}^{-1},...,a_1^{-1}...b_kc^{-1}a_{i+1}^{-1}...b_n\}$$   
$$\mbox{ and } I_{g_2}=\{1,a_1^{-1},... ,a_1^{-1}...b_k,a_1^{-1}...b_kc^{-1},$$
$$a_1^{-1}...b_kc^{-1}b_i^{-1},...,a_1^{-1}...b_kc^{-1}b_i^{-1}...b_{k+1}^{-1}\}, $$
and then consider these sets separately and check them for the assumptions. Repeating this process at the end we get $I_g=\cup_{j=1}^lI_{g_l}$, where each of the sets is full and satisfies the required property, and at least one of them  ends with $g$. For each of these sets $I_{g_j}=\{1,c_1^{-1},c_1^{-1}d_1,...c_1^{-1}d_1c_m^{-1}d_m\}$ define $s_{g_j}$ as above.

Now  the element corresponding to $(g,[\cup_{j=1}^lI_{g_l}])$ is $ (s_{g_1}s_{g_1}^*...s_{g_l}s_{g_l}^*)s_{g_r}$, where $s_{g_r}$ is the element corresponding to the set containing $g$.
The element corresponding to the idempotent $(1,[\cup_{j=1}^lI_{g_l}])$ is $(s_1s_1^*...s_ls_l^*)$. This map is well-defined on  equivalence classes, because it depends on a unique representative of the class. Since any full set is a union of $I_g$, we obtain a well-defined map $\tilde{G}\to S^F$. One can easily verify that this map is also a *-homomorphism. Clearly, the maps $S^F\to\tilde{G}$ and $\tilde{G}\to S^F$ are inverse to one another.   
\end{proof}

For the model of $S^*$ we need to formulate the homomorphism $\beta\colon S^F\to S^*$ in terms of the model of $S^F$. Recall that $\beta$ is given on $S^F$ by equivalence $v_a^*v_a\sim' 1$ for all $a\in S$ and generates the following equivalence for the model of $S^F$.
$$A\cup\{1,g,ga^{-1}\}\sim' A\cup \{1,g\}$$
for any $A\in exp(G)$, $g\in G$, $a\in A$. It follows in particular that
$$\{1,a^{-1}\}\sim' \{1\},\ \{1,g,ga\}\sim' \{1,ga\}$$
Similarly to $\sim$, this new equivalence is stable under taking the union and multiplying by elements of $G$ from the left as defined above. Moreover, $\sim'$ substitutes $\sim$:
$$\{1,g,ga,gab\}\sim'\{1,ga,gab\}\sim'\{1,gab\}$$

Denote by $E'$ the quotient of $exp(G)$ by $\sim'$, and by $[A]$ the equivalence class of $A\in exp(G)$. The product, inverse operation and  partial order on $E'$ are defined similarly to $E$.  Thus \autoref{thm:modelsf} implies the following. 

\begin{cor}\label{cor:modelss} The inverse semigroup $S^*$ is isomorphic to the inverse semigroup
$$\{(g,[A])\in G\times E' \colon [I_g]\leq  [A] \mbox{ for some reduced form of }g\},$$
with the product and inverse operation defined above, idempotent semilattice equal to $E'$, generated by elements
$$v_a=(a,[\{1,a\}]),\ v_a^*=(a^{-1},[\{1\}]).  $$
\end{cor}

\begin{cor}\label{cor:eunitaryall}
	The inverse semigroups $S^F$ and $S^*$ are E-unitary if and only if $S$ is embeddable in a group.
\end{cor}
\begin{proof}
		Due to \autoref{lma:eunitary}, it is sufficient to show the ''if'' part. Suppose $S$ generates a group $G$. Then looking at the model for $S^F$ given by \autoref{thm:modelsf}, any $s\in S^F$ equals
		$(\sigma(s),[A])$ for some set $A\in exp(G)$. Moreover, $s$ is an idempotent if and only if it equals $(1,[A])$. Therefore, $\sigma^{-1}(1)=E(S^F)$. A similar proof works for  $S^*$. 
\end{proof}

\end{pgr}

\section{The reduced C*-algebras of the universal inverse semigroups}

\begin{pgr}  As a consequence of the facts described in the previous section, there is a connection between the representation theories of $S$ and $S^*$. 

Let $P$ be an inverse semigroup. \emph{A *-representation} of $P$ is a homomorphism $\pi$ of $P$ into $B(H)$ such that $\pi(s^*)=\pi(s)^*$ for any $s\in P$. It is clear that each $\pi(s)$ is a partial isometry. We say that $\pi$ is unital if it sends the unit in $P$ to the identity operator. We want to avoid further subtle details concerning the zero element. For this we ask that a *-representation of an inverse semigroup should assign the zero operator to the zero element in $P$ if the latter exists. \end{pgr} 

\begin{pgr} Let $S$ be a left cancellative semigroup. Similarly to the definition given in \cite{Grig}, we say that a representation $\pi$ of $S$ is \emph{an inverse representation} if the set $\pi(S)\cup \pi(S)^*$ generates a  semigroup of partial isometries, i.e. generates an inverse semigroup. It is known that the left regular representation of $S$ is inverse. Note that the well-known requirement of commuting range projections is not sufficient for a representation to be inverse. An example of an abelian semigroup with isometric but non-inverse representation  and a condition for admitting such a representation for general abelian semigroup were given in \cite{A2}.\end{pgr} 

\begin{lma}\label{reprs} There are one-to-one correspondences between inverse representations of $S$ and *-representations of $S^F,$ and between isometric inverse representations of $S$  and unital *-representations of $S^*$.
\end{lma}

\begin{proof}
Given an inverse representation $\pi$ of $S$ and $p\in S$, set $\tilde{\pi}(v_p)=\pi(p)$, and $\tilde{\pi}({v_p}^*)=\pi(p)^*$. Then extend $\tilde{\pi}$ to $S^F$ multiplicatively. Uniqueness of an inverse then follows from Vagner's theorem (\ref{vagner}) and the fact that a product of two partial isometries is a partial isometry if and only if the source projection of the first one commutes with the range projection of the second (see also Proposition 2.3 in \cite{Popov}). Given a *-representation $\tilde{\pi}$ of $S^F$ just set $\pi(s)=\tilde{\pi}(s)$, and the image of an inverse semigroup under *-homomorphism is an inverse semigroup.

The second statement is verified similarly. If $\tilde{\pi}$ is a unital *-representation of $S^*$, since $v_p^*v_p=1$ we get that $\pi(p)=\tilde{\pi}(v_p)$ is an isometry.
\end{proof}

\begin{pgr} Recall the definition of the reduced C*-algebra of an inverse semigroup (see \cite{Paterson} for details). Consider the Hilbert space $\ell^2(P)$ with standard basis $\delta_s, s\in P$. Define \emph{the left regular representation} $V\colon P\to B(\ell^2(P)) $ by  
\begin{equation}\label{invrep}V_s\delta_t=  \left\{
\begin{array}{cc}
	st &\mbox{ if } s^*st=t,\\ 0 &\mbox{ otherwise }
\end{array} \right.\end{equation}

Then $V$ is a *-representation. The reduced C*-algebra of $P$ is $C^*_r(P)=C^*(V_s|\ s\in P)\subset B(\ell^2(P))$. 
\end{pgr} 

\begin{pgr}  Recall the construction of the C*-algebra of a left cancellative semigroup (see \cite{Li}). Let $S$ be a left cancellative semigroup. Consider the Hilbert space $\ell^2(S)$ with standard basis $\delta_p, p\in S$. Define $V_p\in B(\ell^2(S)) $ by 
$$V_p\delta_q=\delta_{pq}$$
for all $p, q\in S$. Then one can check that 
\begin{equation}\label{repcn}V_p^*\delta_q= \left\{
\begin{array}{cc}
	\delta_r &\mbox{ if } q=pr,\\ 0 &\mbox{ otherwise }
\end{array} \right. 
\end{equation}
This is a faithful representation of $S$ by isometries called \emph{the left regular representation} of $S$. The C*-algebra $C^*_r(S)=C^*(V_p|\ p\in S)\subset B(\ell^2(S))$ is the \emph{reduced semigroup C*-algebra of} $S$. \end{pgr}

\begin{lma}\label{lma:lr} The left regular representation of $S$ induces a non-degenerate unital *-representation $V'$ of $S^*$ on $\ell^2(S)$.
\end{lma}

\begin{proof}
As noticed before, the left regular representation of $S$ is inverse, i.e. the semigroup $V(S)$ generated by the set $\{V_p\ |p\in S\}\cup \{V_p^*\ |p\in S\}$ is an inverse semigroup. The reason is that any element of $V(S)$ is a shift operator on the standard basis $\{\delta_p\}$.  Then due to Lemma \ref{reprs}, $V$ induces a *-representation of $S^*$ on $\ell^2(S)$, given on the generators by
$$V'(v_p)\delta_q=\delta_{pq}$$
\end{proof}

Note that $V'$ is not in general faithful, see \autoref{eg:free}. But the image $V'(S^*)$  can be identified with the left inverse hull $I_l(S)$ represented on $\ell^2(S)$.

\begin{lma}\label{lma:restr}  The left regular representation $V$ of the inverse semigroup $S^*$ restricts to a *-representation on $\ell^2(S)\subset \ell^2(S^*)$. This subrepresentation coincides with the *-representation $V'$  and its image is $I_l(S)$. 
\end{lma}

\begin{proof}
The Hilbert space $\ell^2(S)$ is naturally embedded in $\ell^2(S^*)$ by the map given by $\delta_s\mapsto \delta_{v_s}$ for all $s\in S$. It is sufficient to show invariance of $\ell^2(S)$ under the generating operators. For $s,t\in S$  due to (\ref{invrep}) and equality $v_s^*v_s=1$, we have
$$V(v_s)\delta_{v_t}=\delta_{v_sv_t}=\delta_{v_{st}}$$ 
Before checking the same for $v_s^*$, let us show that $v_sv_s^*v_t=v_t$ if and only if $t=sr$ for some $r\in S$. The implication ``$\Leftarrow$'' is obvious. 

Now suppose $t\neq sr$ for any $r\in S$. 
Then using  \autoref{lma:lr}, we have  $V'(v_t)\delta_1\neq\delta_{sr}$ for any $r\in S$. Since by relation (\ref{repcn}) operator $V'(v_s)V'(v_s^*)$ is a projection onto a closed linear span of $\{\delta_{sr}\ | r\in S\}$, we have $$V'(v_sv_s^*v_t)\delta_1=V'(v_s)V'(v_s^*)V'(v_t)\delta_1=0\neq V'(v_t)\delta_1$$
We conclude that $v_sv_s^*v_t\neq v_t$. Hence, using (\ref{invrep}) we obtain
$$V(v_s^*)\delta_{v_t}=\left\{
\begin{array}{cc}
	\delta_{v_r} &\mbox{ if } t=sr,\\ 0 &\mbox{ otherwise }
\end{array} \right. $$ 
We see that $\ell^2(S)$ as a subspace in $\ell^2(S^*)$ is invariant under all $V(v_s)$, $V(v_s^*)$ an therefore under the whole image $V(S^*)$.  Moreover, using identification $ \delta_{v_t}\leftrightarrow\delta_t$ we get $V|_{\ell^2(S)}=V'$.
\end{proof}

\begin{lma}\label{lma:short} The C*-algebra $C^*_{V'}(S^*)$ is isomorphic to $C^*_r(S)$ and the following short exact sequence holds.
\begin{equation}  0\longrightarrow J_r \longrightarrow C^*_r(S^*)\longrightarrow C^*_r(S)\longrightarrow 0\end{equation}
where $J_r$ is the kernel of restriction of  the left regular representation of $S^*$  onto $\ell^2(S)$.
\end{lma}

\begin{proof} Due to \autoref{lma:restr} and \autoref{lma:lr}, $V'$ can be viewed as a restriction of the left regular representation of $S^*$ onto $\ell^2(S)$, and at the same time as a *-representation of $S^*$ induced by the left regular representation $V_S$ of $S$. So, viewing $C^*_r(S)$ as a C*-subalgebra in $B(\ell^2(S^*))$ generated by $V'(S^*)$ we obtain  $C^*_r(S^*)/J_r\cong C^*_r(S)$. 
\end{proof}

%Recall that any *-closed subset $X$ in a C*-algebra $A$ generates a closed *-ideal $J=\overline{(A\cdot X\cdot A)}^{||\cdot ||}$.

%\begin{rmk} Note that due to the faithfulness of the left regular representation of $S^*$ (see Theorem 2.1.1 in \cite{Paterson}), one can consider relations $R_2$ and $R_3$ on $V(S^*)$ induced from $S^*$. Then $J_r$ is a closed ideal generated by $ \{ a-b|\ a\sim_{R_2\cup R_3}b\}$. This describes the C*-algebra $C^*_r(S)$ as a C*-algebra generated by the image of $I_l(S)$ under its representation on $\ell^2(S)$.\end{rmk}

\begin{pgr}\label{repg}

In the case of a semigroup $S$ embeddable in a group we can compute the left regular representations of $S^F$ and $S^*$. Namely, both representations are decomposable into a direct sum of representations, each of which can be realized by a representation on the group $G$, generated by $S$.
 For this we use the models for $S^F$ and $S^*$ given in \autoref{thm:modelsf} and \autoref{cor:modelss} and notations therein.

Recall that in this case the maximal group homomorphic image of $S^F$ equals $G$, $\sigma$ denotes the homomorphism $S^F\to G$. For any element $s=v_{a_1}^*v_{a_2}...v_{a_{n-1}}^*v_{a_n}\in S^F$, 
$$I_s=\{1,\sigma(v_{a_1}^*),\sigma(v_{a_1}^*v_{a_2}),...,\sigma(s)\}\subset G.$$
 This set depends on the form in which $s$ is written. But all forms of $s$ give equivalent sets in the class $[I_s]$. In the model for $S^F$, $s=(\sigma(s),[I_s])$ with idempotent $ss^*$ corresponding to $[I_s]$.
 
 Define for $s\in S^F$, $g\in G$
 $$G_s=\{g\in G\colon \mbox{ there exist } g_1,g_2\in I_s \mbox{ such that } g_1\leq g\leq g_2\}\subset G,$$
$$L_s=\{t\in S^F\colon I_{t^*}\subset G_s\}=\{t\in S^F\colon I_{t}\subset \sigma(t)G_s\}\subset S^F,$$
$$L_{s,g}=\{t\in S^F\colon I_t\subset \sigma(t)g\sigma(s)^{-1}G_s\}.$$

Suppose $g\in G$, $g_1,g_2\in A$, $g_1\leq g\leq g_2$ and $A\sim I_s$ for some fixed form of $s$. Then by definition of $\sim$, there exist $g_1',g_2',g_1'',g_2''\in I_s$ such that 
$$g_1'\leq g_1\leq g_2',\ g_1''\leq g_2\leq g_2''$$
Therefore, $g_1'\leq g\leq g_2''$ and $g\in G_s$. This shows why  $G_s$ does not depend on the form of $s$ and is well-defined. Similarly the sets $L_s$ and $L_{s,g}$ are well-defined.

For $s_1,s_2\in S^F$ define 
$$s_1\sim' s_2 \mbox{ if } I_{s_1^*}\sim I_{s_2^*}.$$
This is an equivalence relation on $S^F$ which does not generate a semigroup congruence. We denote by $[s]$ an equivalence class of $s\in S^F$ and by $Q$ the set of equivalence classes.

\begin{thm}\label{thm:sum}
Let $S$ be embeddable in a group, denote by $G$ the group generated by $S$. The left regular representation $V$ of $S^F$ is isomorphic to $\bigoplus_{[s]\in Q}\pi_{s}$, where $\pi_{s}$ are unital *-representations of $S^F$ on $\ell^2(G_s^{-1}\sigma(s))\subset \ell^2(G)$ defined by
\begin{equation}\label{lsg}
\pi_s(t)\delta_g= \left\{
\begin{array}{cc}
\delta_{\sigma(t)g}  &\mbox{ if } t\in L_{s,g},\\ 0 &\mbox{ otherwise }
\end{array} \right. 
\end{equation}    
\end{thm}

\begin{proof}
For any inverse semigroup $P$ and any $s_1,s_2\in P$ we have
$s_1^*s_1=s_2^*s_2 $  if and only if there exists $t\in P$  such that 
$$ts_1=s_2 \mbox{ and }t^*ts_1=s_1.$$
Therefore, for any $s\in P$, the space $$\ell^2(\{p\in P\colon p^*p=s^*s\})$$
is invariant under the left regular representation of $P$.

Now fix an element $s\in S^F$. Using the model for $S^F$,  $s^*s=(1,[I_{s^*}])$. Hence, $p\sim' s$ if and only if there exists $t\in S^F$ such that $ts=p$ and $t^*ts=s$. It follows that $\ell^2([s])$ is an invariant subspace of $\ell^2(S^F)$ under the representation $V$.

If $t^*ts=s$, then $[I_{t^*}]=[\sigma(t)^{-1}I_t]\leq [I_s]$. It follows that for any fixed form of $t$ and $s$ we have $I_{t^*}\subset G_s$, i.e. $t\in L_s$. This implies $\sigma(t)^{-1}\in G_s$ and
$$ts=(\sigma(t),[I_t])(\sigma(s),[I_s])=(\sigma(ts),[I_t\cup \sigma(t)I_s])=(\sigma(ts),\sigma(t)[I_s]. $$
Hence, if $t\in L_s$, the product $ts$ depends only on $\sigma(t)$.

And due to \autoref{thm:modelsf} for any $g\in G_s^{-1}$ we have $[I_s]=[I_s\cup\{g^{-1}\}]$ and there exists an element $t\in S^F$ such that $[I_{t^*}]\leq [I_s] $ and $\sigma(t)=g$. 

Therefore, $\sigma$ gives a bijection between $[s]\subset S^F$ and $G_s^{-1}\sigma(s)\subset G$. Denote by $\alpha_s$ the corresponding isomorphism between $\ell^2([s])$ and $\ell^2(G_s^{-1}\sigma(s))$. 

Repeating calculations as above, for any $r\in S^F$ and $ts\in [s]$, $g=\sigma(t)$ we have for any $I_{r^*}$ corresponding to a fixed form of $r$ that
$$r^*r(ts)=ts  \Longleftrightarrow I_{r^*}\subset g\sigma(s)^{-1}G_s \Longleftrightarrow r\in L_{s,g}.$$

Consequently, defining the representation $\pi_s$ by formula (\ref{lsg}) we obtain for any $t,x\in S^F$:
$$ \alpha_s(V_t\delta_x)=\pi_s(t)\delta_{\sigma(x)}.$$
Thus, $\alpha_s$ is an isomorphism between the restriction of $V$  onto $\ell^2([s])$ and $\pi_s$.
\end{proof}
\end{pgr}

\begin{pgr}\label{pgr:repsum}
	
A similar result holds for the semigroup $S^*$. We use $\sigma$ to denote the homomorphism $S^*\to G$. Recall the equivalence on $S^F$ defining $S^*$. For any $g\in G, a\in S$:
$$v_a^*v_a\sim 1,\ \{1,g,ga \}\sim' \{1,ga\}.$$
For any $s\in S^*$ define  
$$P_s=\{t\in S^*\colon I_{t^*}\sim' I_{s^*} \}.$$ 
Denote by $R$ the set of all sets $P_s\subset S^*$.
Also define
$$D_s=I_s\cdot S^{-1}=\{ ga^{-1}\colon g\in I_s, a\in S\}\subset G$$

\begin{thm}	\label{thm:sums} 
	Let  $G$ be a group generated by a semigroup $S$. The left regular representation $V$ of $S^*$ is isomorphic to $\bigoplus_{P_s\in R}\pi_{s}$, where $\pi_{s}$ are unital *-representations of $S^*$ on $\ell^2(D_s^{-1}\sigma(s))\subset \ell^2(G)$ defined by
	\begin{equation}\label{rs}
	\pi_s(t)\delta_g= \left\{
	\begin{array}{cc}
	\delta_{\sigma(t)g}  &\mbox{ if } I_{t^*}\subset g\sigma(s)^{-1}D_s,\\ 0 &\mbox{ otherwise }
	\end{array} \right. 
	\end{equation}    
\end{thm}

\begin{proof}
	
The proof is similar to \autoref{thm:sum}, the difference is in the quasi order $\leq$ on $exp(G)$ under the equivalence $\sim'$. For $A,B\in exp(G)$, $[A]\leq [B]$ if and only if for any $g\in A$ there exists $a\in S$ such that $ga\in B$.

 Fix an element $s\in S^*$ and $I_s$ for one of its forms. Then for any $t\in S^*$ we have $[I_{t^*}]\leq [I_s]$ if and only if $I_{t^*}\subset D_s$ for any form of $t$. This implies that $\sigma(t)\in D_s^{-1}$ and $\sigma(ts)\in D_s^{-1}\sigma(s)$. As in \autoref{thm:sum}, when $t$ satisfies $t^*ts=s$, the product $ts$ depends only on $\sigma(t)$. Consequently, 
 $$P_s=\{ts\colon t^*ts=s\}\subset S^*$$
 is bijective to $D_s^{-1}\sigma(s)\subset G$, with bijection implemented by $\sigma$. 
 
 Let $g\in D_s^{-1}\sigma(s)$. Then we can find $t\in S^*$ such that $t^*ts=s$ and $\sigma(t)=g\sigma(s)^{-1}$. For any $r\in S^*$ we have $r^*rts=ts$ if and only if 
 $$I_{r^*}\subset D_{ts}=I_{ts}S^{-1}=\sigma(t)D_{s}=g\sigma(s)^{-1}D_s.$$
 Hence, the statement of Theorem follows.\end{proof}

\end{pgr}

\section{The full C*-algebras of the universal inverse semigroups}

\begin{pgr} Let $P$ be an inverse semigroup. Consider the space $\ell^1(P)$, define product and involution:
$$(\sum_{s\in P}a_s\delta_s)(\sum_{t\in P}b_t\delta_t)=(\sum_{s,t\in P}a_sb_t\delta_{st}) $$
$$(\sum_{s\in P}a_s\delta_s)^*=\sum_{s\in P}\overline{a_s}\delta_{s^*}$$
Then $\ell^1(P)$ is a Banach *-algebra. Any *-representation of $P$ extends to a *-representation of $\ell^1(P)$ and the converse is true. \emph{The full semigroup C*-algebra} $C^*(P)$ of $P$ is the completion of $\ell^1(P)$ under the supremum norm over all *-representations of $P$ (\cite{Paterson}).

 \end{pgr}
 
\begin{pgr}\label{pgr:fullcancel}
For a left cancellative semigroup $S$  let $\mathcal{J}$ be the set of all constructible right ideals in $S$ (see \autoref{pgr:hull}). Define an inverse semigroup $W$ generated by isometries $w_s$, $s\in S$ and projections $e_X$, $X\in \mathcal{J}$  satisfying the following relations for all $s,t\in S$, $X,Y\in\mathcal{J}$:
$$ w_{st}=w_sw_t,\ w_se_Xw_s^*=e_{sX}, $$
$$e_S=1,\ e_\emptyset=0,\ e_{X\cap Y}=e_X e_Y$$
The universal C*-algebra generated by $W$ is \emph{the full semigroup C*-algebra} of $S$, denoted $C^*(S)$ (\cite{Li}). 
\end{pgr}

\begin{lma}\label{lma:surj} There exists a surjective *-homomorphism $C^*(S^*)\to  C^*(S)$.
\end{lma}
\begin{proof} 
Clearly, $W$ is an inverse semigroup. By Corollary 2.10 in \cite{Li}, $W$ is generated by isometries $\{w_s|\ s\in S\}$, i.e. the relations defining $W$ can be formulated in terms of $w_s$. Namely, for $X=a_{n}^{-1}(a_{n-1}... a_{2}^{-1}(a_1 S))$ we have 
$$e_X=w_{a_{n}}^*w_{a_{n-1}}... w_{a_{2}}^*w_{a_1}(w_{a_{n}}^*w_{a_{n-1}}... w_{a_{2}}^*w_{a_1})^*$$
Then $W$ is defined by requirements that idempotent monomials are equal when corresponding ideals in $S$ are equal as sets, and that the product of ideals is respected.  Imposing the same relations on $S^*$ denoted $R_5$, we obtain that the quotient $S^*/R_5$ equals $W$. Then this semigroup *-homomorphism extends to a surjective *-homomorphism  $C^*(S^*)\to  C^*(S)$.
 \end{proof}

\begin{rmk}\label{rmk:r5} In terms of \autoref{pgr:hull} the relations $R_5$ consist of $R_2$ and the first half of $R_3$. Therefore, the *-homomorphism $\gamma\colon S^*\to I_l(S)$ in \autoref{lma:allfree} factors through $S^*\to W$ defined above. 
\end{rmk}

\begin{pgr} As pointed out in the Section 2.5 of \cite{Li},  the construction of the full semigroup C*-algebra is not functorial, i.e. not every semigroup morphism of two cancellative semigroups extends to *-homomorphism of their C*-algebras. This is demonstrated further in \autoref{eg:free}. The reason lies in the construction of the inverse semigroup $W$ in the definition of $C^*(S)$. As we will see now, the map $S\to C^*(S^*)$, on the contrary, is functorial.

\begin{prp}\label{prp:functor} Let $\phi\colon S\to T$ be a semigroup morphism between two left cancellative semigroups. Then $\phi$ extends to a *-homomorphism $\tilde{\phi}\colon C^*(S^*)\to C^*(T^*)$.
\end{prp}
\begin{proof} Due to Theorem 4.8 in \cite{Schein2}, the imbedding $S\hookrightarrow S^*$ is functorial, so $\phi$ induces a *-homomorphism $S^*\to T^*$. By the universal property of inverse semigroup C*-algebras we obtain the required *-homomorphism between the full C*-algebras.
\end{proof}
\end{pgr}

\begin{thm} For any left cancellative semigroup $S$ we have the following commutative diagram in which every line is a short exact sequence and vertical maps are surjective *-homomorphism.
\begin{equation}\label{diagram} 
\begin{CD} 0 @>>> J @>>>C^*(S^*) @>>> C^*(S) @>>> 0\\ 
    @. @.   @VVV                         @VVV               @. \\
0 @>>> J_r @>>>C^*_r(S^*) @>>> C^*_r(S) @>>> 0\end{CD}
\end{equation}
Here $J$ is a closed ideal in $C^*(S^*)$ generated by $\{ x-y| \ x,y\in S^*,\ x  \approx_{R_5} y \}$, and $J_r$ is defined in \autoref{lma:short}.  \end{thm}
\begin{proof}
The upper short exact sequence is given by \autoref{lma:surj}, the lower one was proved in \autoref{lma:short}. The vertical maps are the left regular representations of $S^*$ and $S$ respectively. The commutativity of the diagram then follows from the fact that all these *-homomorphisms are induced by *-homomorphisms of the corresponding inverse semigroups.\end{proof}

\begin{rmk} Due to \autoref{rmk:r5}, $C^*_r(S)$ is a C*-algebra generated by a *-representation of $W$ (in the same way as $S^*$ and $I_l(S)$) on $\ell^2(S)$, which may be not the same as the left regular representation of $W$ and may be not faithful. Therefore in general $C^*_r(S)$ is not isomorphic to $C^*_r(W)$.
\end{rmk}

\begin{cor} Using the results of \cite{Norling}, one has the following commutative diagram where each map is a surjective *-homomorphism:
$$\begin{CD} C^*(S^*) @>>> C^*(S) @>>> C^*(I_l(S))\\ 
 @VVV              @VVV                         @VVV \\
C^*_r(S^*) @>>> C^*_r(S) @<<< C^*_r(I_l(S))\end{CD}
$$
\end{cor}

\begin{rmk} Note that the last arrow is reversed due to the fact that $C^*_r(S)$ is generated by a subrepresentation of the left regular representation of $S^*$ (or $I_l(S)$). Note also that for  inverse semigroups we consider only 0-preserving representations, so that  in our notation $C^*(I_l(S))$ is the same as $C^*_0(I_l(S))$ used in \cite{Norling}. \end{rmk}

\begin{pgr}
We denote $G=G(S^*)$ the maximal group homomorphic image of $S^*$ defined in \autoref{pgr:invsem}. Then the semigroup homomorphism $\sigma\colon S^*\to G$ extends to a surjective *-homomorphism $C^*(S^*)\to C^*(G)$. This surjection is a quotient homomorphism by a closed ideal $I$ in $C^*(S^*)$ generated by the set $\{1-p| \ p\in S^*,  p \mbox{ is an idempotent } \}$. And using the definition of $S^*$ one can see that this ideal is generated by a smaller set $\{1-v_sv_s^*|\ s\in S\}$. By Proposition 1.4 of \cite{Duncan} with a proof in Proposition 4.1 of \cite{Paterson2}, the *-homomorphism $P\to G(P)$ extends to a surjective *-homomorphism of the reduced C*-algebras. Therefore we obtain the following statement.\end{pgr}

\begin{prp}
Let $S$ be a left cancellative semigroup and $G$ the maximal group homomorphic image of $S^*$. Then the following  diagram is commutative, every line is a short exact sequence and vertical maps are surjective *-homomorphism.
\begin{equation}\label{diagram2} 
\begin{CD} 0 @>>> I @>>>C^*(S^*) @>>> C^*(G) @>>> 0\\ 
    @. @VVV   @VVV                         @VVV               @. \\
0 @>>> I_r @>>>C^*_r(S^*) @>>> C^*_r(G) @>>> 0\end{CD}
\end{equation}
Here $I$ is a closed ideal in $C^*(S^*)$ generated by the set $\{1-v_sv_s^*|\ s\in S\}$ and $I_r$ is its image under the left regular representation of $S^*$.\end{prp}

\begin{rmk} There is one more ideal in $C^*(S^*)$ worth to be mentioned, namely a closed ideal $C$ generated by all commutators. Clearly, $C$ contains the ideal $I$, because $v_s^*v_s=1$. In the case of an abelian semigroup $S$, these ideals coincide. In the case $S=\mathbb{Z}_+$, one can easily verify $C^*(S)=C^*(S^*)$ is the Toeplitz algebra and then the ideal $C$ is isomorphic to the ideal of compact operators.  In the general case $C^*(S^*)/C\cong C^*(G_{ab})$, where $G_{ab}$ is an abelianization of $G$.   \end{rmk}

\begin{eg}\label{eg:free}
For any natural number $n$ let $\mathbb{F}_n^+$ denote the free monoid on $n$ generators $a_1,..., a_n$ with the empty word $e$. Let us compare $C^*(\mathbb{F}_n^+)$ (see \cite{Li2} for a detailed description) with $C^*({\mathbb{F}_n^+}^*)$. 
We have $\mathbb{F}_n^+= X_1\sqcup X_2 \sqcup... X_n\sqcup\{e\}$, where $X_i=a_i\mathbb{F}_n^+$ for $i=1,...,n$ are constructible right ideals in $\mathbb{S}_n$ (see \autoref{pgr:hull}). Therefore in  $C^*(\mathbb{F}_n^+)$ for $i\neq j$ we have 

\begin{equation}\label{freemon} v_{a_i}v_{a_i}^*v_{a_j}v_{a_j}^*=e_{X_i}e_{X_j}=e_\emptyset=0,\end{equation}
 which implies $v_{a_i}^*v_{a_j}=0$. The same holds in  $C^*_r(\mathbb{F}_n^+)$, $V_{a_i}^*V_{a_j}=0$.
Hence, the maximal group homomorphic image of $I_l(\mathbb{F}_n^+)$ is trivial. It follows that there is no canonical surjective *-homomorhism from $C^*(\mathbb{F}_n^+)$ onto $C^*(G)$. 

Moreover, knowing that every semigroup is a homomorphic image of a free monoid, one would expect the same to hold for the corresponding semigroup C*-algebras, but this is not true. To see this consider the semigroup $S=\mathbb{Z}_+\times\mathbb{Z}_+$, which is an abelian semigroup generated by $t_1=(1,0)$ and  $t_2=(0,1)$. Then there is a canonical homomorphism $\phi\colon\mathbb{F}_2^+\to S$ sending $a_1\mapsto t_1$, $a_2\mapsto t_2$. And we have $(1,0)S\cap (0,1)S=(1,1)S$, hence $v_{t_1}v_{t_1}^*v_{t_2}v_{t_2}^*=v_{(1,1)}v_{(1,1)}^*\neq 0$. Consequently, due to (\ref{freemon}) homomorphism $\phi$ does not extend to a *-homomorphism $C^*(\mathbb{F}_2^+)\to C^*(S)$. 

In the semigroup ${\mathbb{F}_n^+}^*$ (and its C*-algebra) the product $v_{a_1}v_{a_1}^*v_{a_2}v_{a_2}^*$ is a non-zero idempotent which allows the desired homomorphism to exist. And in general, the semigroup ${\mathbb{F}_n^+}^*$ has no zero, and its maximal group homomorphic image is the free group $\mathbb{F}_n$. So, from the diagram (\ref{diagram2}) we get the following short exact sequence:
$$0\to I\to C^*({\mathbb{F}_n^+}^*)\to C^*(\mathbb{F}_n).$$ 
The same holds for the reduced C*-algebras.
 
Since every semigroup is a homomorphic image of a free monoid using \autoref{prp:functor} we deduce that for any $n$-generated cancellative semigroup $S$ the C*-algebra $C^*(S^*)$ is a homomorphic image of $C^*({\mathbb{F}_n^+}^*)$. 

There is a natural quotient for $C^*({\mathbb{F}_n^+}^*)$ in the case $n=2$ and $n=3$. Consider a closed two-sided ideal $I_n$ generated by the element $p=1-\sum_i v_{a_i}v_{a_i}^*$. Multiplying $p$ by idempotents $v_{a_j}v_{a_j}^*$ one by one we obtain that $I_n$ contains every $v_{a_i}v_{a_i}^*v_{a_j}v_{a_j}^*$ and so contains $v_{a_i}^*v_{a_j}$ for all $i\neq j$. Consequently, the quotient $C^*({\mathbb{F}_n^+}^*)/I_n$ is canonically isomorphic to the Cuntz algebra $\mathcal{O}_n$. The quotient homomorphism can also be seen as the composition $C^*({\mathbb{F}_n^+}^*)\to C^*(\mathbb{F}_n^+)\to \mathcal{O}_n$, where the last map is given in Paragraph 8.2 of \cite{Li2}. In order to get the same in the case $n\geq 4$ one has to add elements $v_{a_i}^*v_{a_j}$ to the ideal $I_n$.
\end{eg}

\begin{pgr}\label{pgr:independence} Many of the results in \cite{Li} and \cite{Li2} hold only under the assumption of independence of the constructible right ideals.  According to \cite{Li}, for a cancellative semigroup $S$ the set of constructible right ideals $\mathcal{J}$ is said to be independent, if $X=\cup_{i=1}^nX_i$ for $X, X_1,..., X_n\in\mathcal{J}$ implies $X=X_i$ for some $i$.

If some of the constructible right ideals are not independent, then the images of  corresponding idempotents under the left regular representation of $S$ are not linearly independent. In such case the diagonal subalgebra $D=C^*(e_X\colon X\in\mathcal{J})$ in $C^*(S)$ is not isomorphic to  its image under the left regular representation.

Let us illustrate such situation on a semigroup $S=\mathbb{Z}_+\setminus \{1\}=\{0,2,3,...\}$ with the usual addition operation. In the notation (\ref{aa}), (\ref{aca}) we have:
$$2+S=\{2,4,5,6,...\},\quad 3+S=\{3,5,6,7,...\},\quad (-3)+2+S=\{2,3,4,...\}.$$
It follows that $(2+S)\cup (3+S)=(-3)+2+S$. Therefore, in $C^*_r(S)$ we have 
$$V_2V_2^*+V_3V_3^*-V_2V_2^*V_3V_3^*=V_3^*V_2V_2^*V_3.$$

But this problem does not exist for the algebra $C^*_r(S^*)$. Due to \cite{Wordin} the left regular representation of an inverse semigroup $P$ extends to a faithful representation of $\ell^1(P)$ on $\ell^2(P)$. Hence, images of elements of $S^*$ under the left regular representation are all linearly independent. The reason is that there are enough vectors in the basis of $\ell^2(S^*)$ to differ elements of $S^*$, unlike the subspace $\ell^2(S)$. In the same way the semigroup elements of $S^*$ are independent in $C^*(S^*)$ and in $C^*(S)$, since these are full semigroup C*-algebras of inverse semigroups ($S^*$ and $W$ respectively). 
\end{pgr}

We end this section with a note on nuclearity of semigroup C*-algebras in the case when $S$ is abelian. This is a direct consequence of Murphy's \cite{Murphy}. 
\begin{prp}
For an abelian cancellative semigroup $S$ the algebras $C^*(S^*)$, $C^*_r(S^*)$, $C^*(S)$, $C^*_r(S)$, $C^*(I_l(S))$, $C^*_r(I_l(S))$ are nuclear.
\end{prp}
\begin{proof}
All of the mentioned C*-algebras in this case are generated by commuting isometries with commuting range projections, so we can apply Theorem 4.8 of \cite{Murphy}.
\end{proof}

\section{Amenability and nuclearity}\label{secamenab}
	
	The classical definition of amenability is common for all semigroups  \cite{Paterson3}, we recall it further. Let $P$ be a semigroup. Right action of $P$ on $\ell^\infty(P)$ is given by
	$$\phi t(x)=\phi(tx),$$
	where $\phi\in \ell^\infty(P)$, $t,x\in P$. A mean $m$ on $\ell^\infty(P)$ is called \emph{left invariant} if $m(\phi t)=m(\phi)$ for all $t\in P$, $\phi\in \ell^\infty(P)$. The semigroup $P$ is \emph{left amenable} if there exists a left invariant mean. 
	For an inverse semigroup left amenability, right amenability and (two-sided) amenability are equivalent. A result of Duncan and Namioka \cite{Duncan2} states that \emph{ an inverse semigroup is amenable if and only if its maximal group homomorphic image is amenable}. The following result of \cite{Milan2} connects amenability of inverse semigroup with the weak containment property.

			\begin{thm}\label{thm:milan2} (D.~Milan \cite{Milan2})
	Let $P$ be an E-unitary inverse semigroup. Then the following statements are equivalent
	\begin{enumerate}
		\item $P$ is amenable,
		\item $P$ satisfies the weak containment property, i.e. $C^*(P)=C^*_r(P)$,
		\item the maximal group homomorphic image $G(P)$ is amenable.
	\end{enumerate}
	  \end{thm}
	
	\begin{cor}\label{cor:amenable} Let $S$ be embeddable in a group. Then the following conditions are equivalent and imply that $C^*_r(S)$ and $C^*(S)$ are nuclear:
		\begin{enumerate}
			\item the group $G$ generated by $S$ is amenable,
			\item $S^*$ is amenable,
			\item $S^F$ is amenable,
			\item $S^*$ has the weak containment property, i.e. $C^*(S^*)=C^*_r(S^*)$,
			\item $S^F$ has the weak containment property,
				\item $C^*_r(S^*)$ is nuclear,
				\item  $C^*_r(S^F)$ is nuclear.
			\end{enumerate}
	\end{cor}
	\begin{proof} By \autoref{cor:eunitaryall}, inverse semigroups $S^F$ and $S^*$ are E-unitary and $G(S^F)=G(S^*)$ coincides with the group generated by $S$. Applying \autoref{thm:milan2}, we get equivalence of all the conditions 1)-5) above. By Proposition 6.6. in \cite{Milan3}, the maximal group homomorphic image $G(P)$ of an inverse semigroup $P$ is amenable if and only if the universal groupoid of $P$ is amenable, which is equivalent to nuclearity of its reduced C*-algebra by Theorem 5.6.18 in \cite{Brown}. Hence, we obtain equivalence of 6) and 7) to 1).	Finally, conditions 6) and 7) and the diagram (\ref{diagram}) imply that $C^*_r(S)$ and $C^*(S)$ are nuclear.
		\end{proof}
		
	\begin{cor}\label{cor:amenable2} Let $S$ be a left amenable cancellative semigroup. Then $S^*$ and $S^F$ are amenable inverse semigroups.
	\end{cor}
	\begin{proof} By Proposition 1.27 in \cite{Paterson3}, every left amenable cancellative semigroup embeds in a group $G$, such that $G=SS^{-1}$ and $G$ is amenable. Hence, condition 1) in \autoref{cor:amenable} is satisfied.		\end{proof}
	 
	 \begin{rmk}
	 	Nuclearity of the C*-algebra $C^*_r(S)$ does not imply nuclearity of $C^*_r(S^*)$ (or $C^*_r(S^F)$). The counterexample is given by the free monoid $\mathbb{F}_n^+$, the subsemigroup generating the free group $\mathbb{F}_n$. By \autoref{eg:free}, there exist surjective *-homomorphisms $C^*({\mathbb{F}_n^+}^*)\to C^*(\mathbb{F}_n)$ and $C^*_r({\mathbb{F}_n^+}^*)\to C^*_r(\mathbb{F}_n)$. Hence, $C^*_r({\mathbb{F}_n^+}^*)$ is not nuclear. This is natural considering the fact that $\mathbb{F}_n^+$ is not amenable. Nevertheless,  as shown by Nica in \cite{Nica}, the C*-algebra $C^*(\mathbb{F}_n^+)$ is nuclear.	 \end{rmk}

We finish this section with a result connecting directly amenability of $S^*$ to amenability of $S$, analogous to Proposition 1.27 in \cite{Paterson3}.

\begin{prp}\label{prp:samen}
Let $S$ be a left cancellative left amenable semigroup. Then $S^*$ is an amenable inverse semigroup.
\end{prp}
\begin{proof}
We consider $S$ embedded in $S^*$ by a map $a\in S\mapsto v_a\in S^*$ and $\ell^\infty(S)\subset \ell^\infty(S^*)$. Let $m_0$ be a left invariant mean on $\ell^\infty(S)$ and define a mean on $\ell^\infty(S^*)$ by $m(\phi)=m_0(\phi|_S)$, where for any $x\in S$ we put
$$\phi|_S(x)=\phi(v_x).$$
Since $S^*$ is generated by $S$, it is sufficient to check left invariance of $m$ on generators $v_a$, $v_b^*$ for all $a,b\in S$. Take $\phi\in \ell^\infty(S^*)$ and $a\in S$ and calculate using left invariance of $m_0$:
$$m(\phi v_a)=m_0((\phi v_a)|_{S})=m_0(\phi|_S a)=m_0(\phi|_S)=m(\phi),$$
$$m(\phi v_a^*)=m_0((\phi v_a^*)|_{S})=m_0((\phi v_a^*)|_S a)=m_0((\phi v_a^*v_a)|_S ) =m_0(\phi|_S)=m(\phi)  $$
Hence, $m$ is a left invariant mean on $\ell^\infty(S^*)$.\end{proof}

\section{Crossed products of universal inverse semigroups}\label{seccross}

\begin{dfn} Let $P$ be an inverse semigroup. \emph{An  action} $\alpha$ of $P$ on a space $X$ is a *-homomorphism $P\to \mathcal{I}(X)$, $\alpha_s\colon D_{s^*}\to D_s$, such that the union of all $D_s$ coincides with $X$.  We call it \emph{unital} if the image of the unit element in $P$ is the identity map on $X$. If $X$ is a locally compact Hausdorff topological space, we require  every $\alpha_s$ to be continuous and $D_s$ to be open in $X$.
\end{dfn}

\begin{lma}\label{lma:actsp} There is a one-to-one correspondence between  actions of $S$ on a space $X$ by injective maps and unital actions of $S^*$ on $X$. 
\end{lma}
\begin{proof} Let $\alpha$ be an action of $S$ on $X$, i.e. $\alpha_s\alpha_t=\alpha_{st}$ for any $s,t\in S$, such that each $\alpha_s$ is injective.  By Lemma \ref{act}, denoting $D_s\subset X$ the image of $\alpha_s$ and defining $\alpha_s^*\colon D_s\to X$ as the inverse of $\alpha_s$, we get a set generating an inverse subsemigroup $S^*_\alpha$ in $\mathcal{I}(X)$. Clearly, in this semigroup the map $\alpha_s^*\alpha_s$ is an identity on $X$ for any $s\in S$. Hence, there is a surjective *-homomorphism $\tilde{\alpha}\colon S^*\to S^*_\alpha$, which gives an action of $S^*$ by partial bijections on $X$. And we see that $\tilde{\alpha}$ is a unital  action of $S^*$ on $X$.

Now suppose $\alpha$ is a unital action of $S^*$ on $X$. Then define $\tilde{\alpha}(s)=\alpha(v_s)$ for all $s\in S$. Then multiplicativity follows immediately. Unitality of $\alpha$ implies that 
$$\tilde{\alpha}(s)^*\tilde{\alpha}(s)=\alpha(v_s^*v_s)=\mathrm{id}.$$
Hence, for every $s\in S$ $\tilde{\alpha}(s)$ is a bijection with the domain equal to $X$.
\end{proof}

\begin{rmk}
The previous lemma holds also in topological setting, i.e. when $X$ is a locally compact Hausdorff topological space and each $\alpha_s$ is continuous.\end{rmk}

\begin{dfn}\label{dyn} By an \emph{injective action with ideal images} $\alpha$ of a left cancellative semigroup $S$ on a C*-algebra $A$ we mean a set of injective *-homomorphisms $\alpha_s$ on $A$ such that for every $s,t\in S$,  $\alpha_{st}=\alpha_s\alpha_t$ and  $\alpha_s(A)$ is a closed two-sided *-ideal in $A$. In this case we say that $(\alpha,S,A)$ is an \emph{injective C*-dynamical system}.

 \emph{A partial automorphism} $\phi$ on a C*-algebra $A$ is a *-isomorphism $\phi\colon J_1\to J_2$, where $J_1$, $J_2$ are closed two-sided *-ideals in $A$.    For a C*-algebra $A$ denote by $\mathcal{I}(A)$ the inverse semigroup of partial automorphisms on $A$, with a product and an inverse map defined similarly to $\mathcal{I}(X)$ (see Section 2).  

\emph{An  action} $\alpha$ of an inverse semigroup $P$ on a C*-algebra $A$ is a *-ho\-mo\-mor\-phism $P\to \mathcal{I}(A)$, $\alpha_s\colon E_{s^*}\to E_s$, such that the union of all $E_s$ coincides with $A$. In this case we say that $(\alpha,P,A)$ is a \emph{C*-dynamical system}. If $P$ has a unit $1$ and $\alpha_1=\mathrm{id}_A$, we say that the action $\alpha$ is \emph{unital}. 

\end{dfn}

\begin{lma}\label{lma:acta} There is a one-to-one correspondence between injective actions of $S$ with ideal images on a C*-algebra $A$  and unital actions of $S^*$ on $A$. 
\end{lma}
\begin{proof}  
For an injective action $\alpha$ of $S$ on $A$, we define for any $s\in S$ the domain $E_{v_s^*}=A$ and the range $E_{v_s}=\alpha_s(A)$ of $\tilde{\alpha}(v_s)=\alpha_s$. For the inverse map we put $\tilde{\alpha}(v_s^*)=\alpha_s^*\colon E_{v_s}\to E_{v_s^*}$.   Following the proof of \autoref{lma:actsp}, we obtain an action of $S^*$ on the underlying space of $A$. Since $\alpha_s$ is a *-homomorphism, the same is true for $\tilde{\alpha}(v_s)$, $\tilde{\alpha}(v_s^*)$ and the products of such maps. Hence, $\tilde{\alpha}$  given by \autoref{lma:actsp} is an action of $S^*$ on the C*-algebra $A$. The reverse statement follows similarly from \autoref{lma:actsp}.
\end{proof}

\begin{rmk}\label{unital} Notice that if $A$ is unital and $\tilde{\alpha}$ is an induced action of $S^*$ on $A$ as in the Lemma,  we may extend $\tilde{\alpha}(v_s^*)$ to a *-endomorphism on $A$ by setting $\tilde{\alpha}(v_s^*)(a)=\alpha_s^*(\alpha_s(1)a)$. But one should remember that this extension is injective only on $E_{v_s}$. Then $\tilde{\alpha}$ is an action of $S^*$ on $A$ by *-endomorphisms.
\end{rmk}

\begin{dfn}\label{covcan} Let $S$ be a left cancellative semigroup with an action $\alpha$ on a C*-algebra $A$. A \emph{covariant representation} (see \cite{Laca}) of the C*-dynamical system $(\alpha,S,A)$ is a pair $(\pi,T)$ in which 
\begin{enumerate}
	\item $\pi$ is a non-degenerate *-representation of $A$ on $H$,
	\item $T\colon S\to B(H)$ is a unital isometric representation of $S$,
	\item the covariance condition  $\pi(\alpha_s(a))=T_s\pi(a)T_s^*$ holds for every $a\in A, s\in S$.
		\end{enumerate}\end{dfn}

\begin{lma}\label{cancelcov} If $(\pi,T)$ is a covariant representation of an injective C*-dynamical system with ideal images $(\alpha,S,A)$ and $A$ is unital, then $T$ is an inverse representation of $S$.
\end{lma}
\begin{proof} It is sufficient to show that all idempotents in the semigroup $S_T$ generated by $T(S)\cup T(S)^*$ commute, i.e. $xx^*yy^*=yy^*xx^*$ for any monomials $x$ and $y$. Obviously, in the notation of \autoref{lma:acta}, $T_sT_s^*=\pi(\alpha_s(1))\in\pi(E_s)$. Due to the Remark \ref{unital} and the covariance condition, $\pi(\tilde{\alpha}(v_s^*)(a))=T_s^*\pi(a)T_s$. Generally, for any monomial $x\in S_T$ and $a\in A$ we have
\begin{equation} \label{a1} xx^*=\pi(\tilde{\alpha}(x)(1)),\end{equation}
\begin{equation} \label{a2} x\pi(a)x^*=\pi(\tilde{\alpha}(x)(a))\end{equation}

For $x=T_s^*$, $s\in S$ using the fact that images of $\tilde{\alpha}$ are ideals we deduce:
\begin{equation} \label{a3} T_s^*\pi(a)= T_s^*\pi(\tilde{\alpha}(s)(1)a)=T_s^*\pi(a\tilde{\alpha}(s)(1))=T_s^*\pi(a)T_sT_s^*,\end{equation}
\begin{equation}\label{a4} \pi(a)T_s=\pi(a\tilde{\alpha}(s)(1))T_s=\pi(\tilde{\alpha}(s)(1)a)T_s=T_sT_s^*\pi(a)T_s \end{equation}
For general type of monomial these formulas can be shown by induction on the length of monomial. Suppose the formulas hold for the length equal $n$ and take $x=yz$, where $y$ has length $n$ and $y$ is a generator. First assume $z=T_s$ and for $a\in A$ calculate
$$\pi(a)x=\pi(a)yT_s=yy^*\pi(a)yT_s=$$  
By (\ref{a2}), we have $y^*\pi(a)y\in\pi(A)$, and due to (\ref{a4})
$$y(y^*\pi(a)y)T_s=yT_sT_s^*y^*\pi(a)yT_s=xx^*\pi(a)x.$$
If $z=T_s^*$, by $T_s^*T_s=1$ we similarly have
$$\pi(a)x=\pi(a)yT_s^*=yy^*\pi(a)yT_s^*=yT_s^*T_sy^*\pi(a)yT_s^*=xx^*\pi(a)x.$$ 
Thus, for any monomial $x$ we have
\begin{equation} \label{a5} \pi(a)x =xx^*\pi(a)x.\end{equation} 

In the same way, splitting $x$ into $T_sy$ or $T_s^*y$ and using (\ref{a3}), we obtain
\begin{equation} \label{a6} x\pi(a) =x\pi(a)x^*x.\end{equation} 
Then for any monomials $x,y$ using that  $yy^*\in\pi(A)$ by (\ref{a1}), we get
$$xx^*yy^*\overset{(\ref{a6})}{=}xx^*yy^*xx^*\overset{(\ref{a5})}{=}yy^*xx^*$$
\end{proof}

\begin{dfn}\label{covinv} Let $P$ be an inverse semigroup with an action $\alpha$ on a C*-algebra $A$. A \emph{covariant representation} (see \cite{Sieben}) of the C*-dynamical system $(\alpha,P,A)$ is a pair $(\pi,T)$ in which 
\begin{enumerate}
	\item $\pi$ is a non-degenerate *-representation of $A$ on $H$,
	\item $T\colon P\to B(H)$ is a unital *-representation of $P$, such that  for every $s\in P$, $T_s^*T_sH=\pi(E_{s^*})H$ and $T_sT_s^*H=\pi(E_{s})H$
	\item the covariance condition  $\pi(\alpha_s(a))=T_s\pi(a)T_s^*$ holds for every $a\in E_{s^*}, s\in P$.
		\end{enumerate}
\end{dfn}

\begin{lma}\label{covrep} Let $\alpha$ be an injective action of a left cancellative semigroup $S$ on a C*-algebra $A$ and $\tilde{\alpha}$ the induced action of $S^*$ on $A$. Then there is a one-to-one correspondence between the covariant representations of $(\alpha, S, A)$ and $(\tilde{\alpha}, S^*, A)$. 
\end{lma}
\begin{proof}

Let $(\pi,T)$ be a covariant representation of $(\alpha, S, A)$ on $H$. By Lemma \ref{reprs} and Lemma \ref{cancelcov}, $T$ induces a *-representation $\tilde{T}$ of $S^*$ on $H$ given by
$$\tilde{T}_{v_s}=T_s,\ \tilde{T}_{v_s^*}=T_{s}^*.$$ 
The condition (\ref{a2}) gives the condition (3) of the Definition \ref{covinv}.

Now we prove condition (2) of Definition (\ref{covinv}).  Let $v\in S^*, a\in E_{v^*}$ and $b=\tilde{\alpha}_v(a)$. Due to covariance condition, we have
\begin{equation} \label{e1}\pi(\alpha_v(a))=\tilde{T}_v \pi(a)\tilde{T}_v^*=\tilde{T}_v\tilde{T}_v^*\tilde{T}_v\pi(a)\tilde{T}_v^*=\tilde{T}_v\tilde{T}_v^*\pi(\alpha_v(a)).\end{equation}
Hence, $\pi(E_v)H\subset\tilde{T}_v\tilde{T}_v^*H$. 

We prove the reverse inclusion by induction on the length of $v$. First suppose $v=v_sw$, where $s\in S$, $w\in S^*$, and assume that the inclusion is proved for $w$.  It implies that for $x\in H$, the vector $\tilde{T}_v\tilde{T}^*_vx=\tilde{T}_{v_s}\tilde{T}_w\tilde{T}_w^*\tilde{T}_{v_s^*}x$ can be approximated by  $\sum_i\tilde{T}_{v_s}\pi(a_i)y_i$ for some $a_i\in E_w$ and $y_i\in H$. Hence, we obtain
$$\tilde{T}_v\tilde{T}^*_vx \approx \sum_i\tilde{T}_{v_s}\pi(a_i)\tilde{T}_{v_s}^*\tilde{T}_{v_s}y_i=\sum_i\pi(\alpha_s(a_i))\tilde{T}_{v_s}y_i\in \pi(E_{v_sw})H$$

Now suppose $v=v_s^*w$ for $s\in S$, $w\in S^*$, and assume that the inclusion is proved for $w$. Similarly the vector $\tilde{T}_v\tilde{T}^*_vx=\tilde{T}_{v_s}^*\tilde{T}_w\tilde{T}_w^*\tilde{T}_{v_s}x$ is approximated by $\sum_i\tilde{T}_{v_s}^*\pi(a_i)y_i$ for some $a_i\in E_w$ and $y_i\in H$.  
Denote by $u_\lambda$ the approximate unit of $A$. Due to the fact that $\pi$ is a non-degenerate representation of $A$, $\pi(u_\lambda)$ converges to the identity operator on $H$ in the strong operator topology, i.e. $y\approx\pi(u_\lambda)y$ for any $y\in H$. Then we obtain
$$\tilde{T}_v\tilde{T}^*_vx=\sum_i\tilde{T}_{v_s}^*\tilde{T}_{v_s}\tilde{T}_{v_s}^*\pi(a_i)y_i\approx \sum_i\tilde{T}_{v_s}^*\tilde{T}_{v_s}\pi(u_\lambda)\tilde{T}_{v_s}^*\pi(a_i)y_i=$$
$$ \sum_i\tilde{T}_{v_s}^*\pi(\alpha_s(u_\lambda))\pi(a_i)y_i=\sum_i\tilde{T}_{v_s}^*\pi(\alpha_s(b_{i,\lambda}))y_i,$$
where $b_{i,\lambda}=\alpha_s(u_\lambda)a_i\in E_{v_s}\cap E_w$. Since $b_{i,\lambda}\in E_{v_s}$, using (\ref{e1}) we get
$$\sum_i\tilde{T}_{v_s}^*\pi(b_{i,\lambda})y_i= \sum_i\tilde{T}_{v_s}^*\pi(b_{i,\lambda})\tilde{T}_{v_s}\tilde{T}_{v_s}^*y_i= \sum_i\pi(\tilde{\alpha}_{v_s}^*(b_{i,\lambda}))\tilde{T}_{v_s}^*y_i,$$
which belongs to $\pi(E_{v_s^*w})H $ due to definition of $E_{v_s^*w}$. Thus, $(\pi,\tilde{T})$ is a covariant representation of $(\tilde{\alpha}, S^*, A)$. 

If $(\pi,\tilde{T})$ is any covariant representation of $(\tilde{\alpha}, S^*, A)$, then $T_s=\tilde{T}_{v_s}$ gives a unital inverse representation of $S$ by Lemma \ref{reprs}. Since $\alpha$ is just a restriction of $\tilde{\alpha}$, the covariance condition also holds.
\end{proof}

\begin{rmk}\label{covrep2} The reverse statement to the previous Lemma also holds. Let $S$ be a left cancellative semigroup and $\alpha$ be an action of the universal inverse semigroup $S^*$ on a C*-algebra $A$, $\tilde{\alpha}$ the induced action of $S$ on $A$. Then there is a correspondence between the covariant representations of $(\alpha, S^*, A)$ and $(\tilde{\alpha}, S, A)$.  The proof is the same as above.
\end{rmk}

\begin{pgr} We now recall the definitions of crossed products by partial automorphisms of inverse semigroups and crossed products by injective actions of cancellative semigroups. For the case of inverse semigroups acting on (in general non-unital) C*-algebras  by partial automorphisms there is a well-established definition, given in \cite{Sieben}, \cite{Exel}, \cite{Exel2}. 

For cancellative semigroups there exist several constructions of the crossed product, by automorphisms (\cite{Li}) and by endomorphisms (\cite{Laca}), both of them for unital C*-algebras. Though an automorphism is a particular case of an endomorphism, the construction of the crossed product by an automorphism is not a particular case of the one by an endomorphism. 

The main difference of the constructions is that in the case of an automorphism the crossed product contains  the whole semigroup C*-algebra, while the crossed product by an endomorphism contains only isometries corresponding to the elements of the semigroup. A connection between crossed products by automorphisms and crossed products by inverse semigroups is given in Proposition 5.7 of \cite{Li2}. It describes the isomorphism between  $A\rtimes^a_\alpha S$  and $(A\otimes D)\rtimes_\beta  \tilde{S}$ in the case when $S$ is a subsemigroup of a group. Here $D$ is the canonical commutative C*-subalgebra in $C^*(S)$, and  $\tilde{S}$ is some specific *-homomorphic image of $S^*$.  Thus, the subject of the present paper is the notion of crossed product by an endomorphism of $S$. 
\end{pgr}

\begin{dfn} Let $\alpha$ be an action of a left cancellative semigroup $S$ on a unital C*-algebra $A$. The crossed product associated to the C*-dynamical system $(\alpha,S,A)$ is a C*-algebra $A\rtimes_{\alpha} S$ with a unital *-homomorphism $i_A\colon A\to A\rtimes_{\alpha} S$ and an isometric representation $i_S\colon S\to A\rtimes_{\alpha} S$  such that 

\begin{enumerate}
	\item $(i_A,i_S)$ is a covariant representation for $(\alpha,S,A)$,
	\item for any other covariant representation $(\pi,T)$ there is a representation $\pi\times T$ of $A\rtimes_{\alpha} S$ such that $\pi=(\pi\times T)\circ i_A$ and $T=(\pi\times T)\circ i_S$,
	\item $A\rtimes_{\alpha} S$ is generated by $i_A(A)$ and $i_S(S)$ as a C*-algebra.
\end{enumerate}
\end{dfn}
As noticed in \cite{Laca}, the crossed product is non-trivial if $S$ acts by injective endomorphisms. $A\rtimes_{\alpha} S$ can be defined as a C*-algebra generated by monomials $aw_s$ and $(aw_s)^*=w_s^*a^*$  where $a\in i_A(A)$, and $w_s=i_S(s)$. The completion is taken with respect to the norm given by the supremum of $||(\pi\times T)(x)||$ over all covariant representations $(\pi,T)$. Obviously, for any covariant representations $(\pi,T)$, the representation $\pi\times T$ of $A\rtimes_{\alpha} S$ is given by 
$$(\pi\times T)(w_t^* a w_s)=T_t^*\pi(a) T_s.$$
Note that since $A$ and $S$ are unital, the elements $w_s$ generate a semigroup isomorphic to $S^*$ and we can change the notation $w_s$ to $v_s$.

\begin{dfn}\label{crossinv} Let $\alpha$ be an action of an inverse semigroup $P$ on a C*-algebra $A$. Denote by $L$ the linear space of finite sums $\sum_{s\in P} a_s\delta_s$ where $a_s\in E_s$ and $\delta_s$ is a formal symbol. Multiplication and involution are defined in the following way. 
$$(a\delta_s)(b\delta_t)=\alpha_s(\alpha_{s^*}(a)b)\delta_{ts},\ (a\delta_s)^*=\alpha_{s^*}(a^*)\delta_{s^*}.$$
For any covariant representation $(\pi, T)$ define a non-degenerate *-rep\-re\-sen\-ta\-tion of $L$:
$$(\pi\times T)(\sum_{s\in P}a_s\delta_s)=\sum_{s\in P}\pi(a_s)T_s$$
 The crossed product $A\rtimes_{\alpha} P$ is the Hausdorff completion of $L$ in the norm 
$$||x ||=\mathrm{sup}_\Pi || \Pi(x) ||,$$
where the supremum is taken over all representations of $L$ of the form $\Pi=\pi\times T$ for all covariant representations $(\pi,T)$.
  \end{dfn}

\begin{thm}\label{thm:tcross} Let $A$ be a unital C*-algebra and let $\alpha$ be an injective action with ideal images of $S$ on $A$, $\tilde{\alpha}$ be an action of $S^*$ on $A$, where one of them is induced by another. Then the crossed product C*-algebras $A\rtimes_\alpha S$ and $A\rtimes_{\tilde{\alpha}} S^*$ are isomorphic. 
\end{thm}
\begin{proof} 
Due to  Lemma \ref{covrep} and  Remark \ref{covrep2}, it is sufficient to give a *-isomorphism of the underlying *-algebras of monomials, $K$ for $A\rtimes_\alpha S$ and $L$ for $A\rtimes_{\tilde{\alpha}} S^*$ respectively.

First note that any monomial in $K$ can be written as $av_{s_1}^*v_{s_2}...v_{s_n}$, where $a\in E_{v_{s_1}^*v_{s_2}...v_{s_n}}$ in the notation of \autoref{lma:acta}. Indeed, using the covariance condition and the assumption that the images of the endomorphisms are ideals, we obtain for any $a\in A$, $s\in S$:

$$av_s=av_sv_s^*v_s=a\alpha_s(1) v_s, $$ 
$$v_s^*a= v_s^*\alpha_s(1)a=v_s^* \alpha_s(\alpha_s^*(\alpha_s(1)a))=\alpha_s^*(\alpha_s(1)a)v_s^*,$$
$$v_sa=\alpha_s(a)v_s.$$
As noticed in the proof of \autoref{lma:acta}, we may assume that $\tilde{\alpha}(v_s^*)(a)=\alpha_s^*(\alpha_s(1)a)$ for any $a\in A$, remembering that $\tilde{\alpha}(v_s^*)$ is an isomorphism only on $E_{v_s}$. Therefore, for any monomial $x=v_{s_1}^*v_{s_2}...v_{s_n}$ and $a\in A$ we obtain
$$xa=\tilde{\alpha}(x)(a)x.$$

Define $\phi\colon K\to L$ on generators by $a\to a\delta_1$, $v_s\to \tilde{\alpha}(v_s)(1)\delta_{v_s}$ and on an arbitrary monomial $x$  and $ a\in E_{x}$ by
$$\phi(ax)=a\delta_x.$$
Extend $\phi$ linearly to $K$. Then clearly, $\phi(K)=L$. To show that $\phi$ is multiplicative, calculate the product of arbitrary monomials $av_{t_1}^*v_{t_2}...v_{t_n}$ and $bv_{s_1}^*v_{s_2}...v_{s_k}$, where $a\in E_{x}$, $b\in E_y$, $x=v_{t_1}^*v_{t_2}...v_{t_n}$, $y=v_{s_1}^*v_{s_2}...v_{s_k}$:
$$axby=a\tilde{\alpha}(x)(b)xy=\tilde{\alpha}(x)(\tilde{\alpha}(x^*)(a) b) xy. $$
Therefore, due to the Definition \ref{crossinv} $$\phi(axby)=\phi(\tilde{\alpha}(x)(\tilde{\alpha}(x^*)(a) b) xy)=\tilde{\alpha}(x)(\tilde{\alpha}(x^*)(a) b)\delta_{xy}=(a\delta_x)(b\delta_y).$$
In the same way we verify that $\phi$ preserves involution:
$$\phi(ax)^*=(a\delta_x)^*=\alpha(x^*)(a^*)\delta_{x^*}=\phi(\alpha(x^*)(a^*)x^*)=\phi(x^*a^*).$$
Thus, $\phi$ is a *-isomorphism onto $L$.
\end{proof}

\begin{cor} There exist an injective action $\tilde{\beta}$ of $S$ on $C^*(E)$, where $E$ is the semigroup  of idempotents in $S^*$, and an action $\tilde{\alpha}$ of $S^*$ on $E_\mathcal{J}=\{e_X| X\in\mathcal{J}\}$. With respect to thess actions,
 $$C^*(S^*)\cong C^*(E)\rtimes_{\tilde{\beta}}S,\ C^*(S)\cong C^*(E_\mathcal{J})\rtimes_{\tilde{\alpha}} S^*.$$ \end{cor}
\begin{proof}
For any inverse semigroup $P$ there exists an action $\beta$ by partial bijections on its subsemigroup $E$ of idempotents. Namely, for $x\in P$, the domain of $\beta_x$ is $D_{x^*}=\{f\colon f=x^*xf\}$, and $\beta_x(f)=xfx^*$. This action extends to an action $\beta$ of $P$ on the commutative C*-algebra $C^*(E)$ and by Proposition 4.11 of \cite{Sieben}, $C^*(P)$ is isomorphic to the crossed product $C^*(E)\rtimes_\beta P$. Take $P=S^*$ and see that $\beta$ is unital and $\beta_{v_s}$ is injective since $v_s$ are isometries for $s\in S$, and the images of the extension of $\beta$ to $C^*(E)$ are closed ideals. Then by \autoref{lma:acta} we get an action $\tilde{\beta}$ of $S$ on  $C^*(E)$ by injective endomorphisms with ideal images. Applying \autoref{thm:tcross} we obtain
$$C^*(S^*)\cong C^*(E)\rtimes_\beta S^*\cong C^*(E)\rtimes_{\tilde{\beta}} S.$$
By Lemma 2.14 in \cite{Li}, $C^*(S)\cong C^*(E_\mathcal{J})\rtimes_\alpha S$, where $\alpha$ is an injective action given on generators by $\alpha_s(e_X)=e_{sX}$ for $s\in S, X\in\mathcal{J}$. Therefore, by \autoref{lma:acta} $\alpha$ generates an action $\tilde{\alpha}$ of $S^*$ on $C^*(E_\mathcal{J})$. Using again \autoref{thm:tcross}, we obtain the required isomorphism.
\end{proof}

\begin{rmk} The isomorphism $C^*(S)\cong C^*(E_\mathcal{J})\rtimes_\alpha S$ of \cite{Li} mentioned above could be deduced directly from the fact that $C^*(S)$ is a C*-algebra of an inverse semigroup $W$, which is a quotient of $S^*$,  using \autoref{thm:tcross}. \end{rmk}

\begin{cor} Let $S$  be a left Ore semigroup generating a group $G=S^{-1}S$ and $\alpha$ be a unital action of $S^*$ on a C*-algebra $A$. Then there exists a unique up to isomorphism C*-dynamical system $(B,G,\beta)$, where $\beta$ is an action of $G$ by automorphisms of a C*-algebra $B$ and an embedding $i\colon A\hookrightarrow B$ such that
\begin{enumerate}
	\item $\beta$ dilates $\alpha$, i.e. $\beta_s\circ i=i\circ \alpha_{v_s}$ for all $s\in S$,
	\item $\cup_{s\in S} \beta_s^{-1}(i(A))$ is dense in $B$,
	\item $A\rtimes_\alpha S^*$ is isomorphic to $i(1)(B\rtimes_\beta G)i(1)$, which is a full corner.
\end{enumerate}
\end{cor}
\begin{proof} Use \autoref{thm:tcross} and apply Theorem 2.4 in \cite{Laca}.
\end{proof}

\begin{pgr} The crossed product with a non-unital C*-algebra is defined using the multiplier algebra (see \cite{Larsen} for details), but the notion of the covariant representation is the same as in Definition \ref{covcan}. For an extendible *-homomorphism between C*-algebras $\phi\colon A\to M(B)$, its unique strictly continuous extension is denoted $\overline{\phi}\colon M(A)\to M(B)$.\end{pgr}

\begin{dfn} Let $(\alpha,S,A)$ be a C*-dynamical system, where $A$ is non-unital and $S$ is a semigroup. A crossed product for $(\alpha,S,A)$ is a C*-algebra $B$ denoted $A\rtimes_{\alpha}S$ with a proper homomorphism $i_A\colon A\to B$ and a semigroup homomorphism $i_S\colon S\to \mathrm{Isom}(M(B))$ such that
\begin{enumerate}
	\item $(i_A,i_S)$ is a covariant representation for $(\alpha,S,A)$,
	\item for any other covariant representation $(\pi,T)$ there is a non-de\-ge\-ne\-rate representation $\pi\times T$ of $A\rtimes_{\alpha} S$ such that $$\pi=(\pi\times T)\circ i_A \mbox{ and } T=(\overline{\pi\times T})\circ i_S,$$
	\item $A\rtimes_{\alpha} S$ is generated by $\{i_A(a)i_S(s)|\ a\in A,\ s\in S\}$ as a C*-algebra.
 
\end{enumerate}

\end{dfn}

\begin{thm}\label{crossall}  Let $A$ be a non-unital C*-algebra and let $\alpha$ be an extendible injective action of $S$ on $A$, $\beta$ be an extendible action of $S^*$ on $A$, where one of them is induced by another. Then the crossed product C*-algebras $A\rtimes_\alpha S$ and $A\rtimes_{\beta} S^*$ are isomorphic. 
\end{thm}
\begin{proof}  

The action $\alpha$ of $S$ on $A$ extends to an action $\overline{\alpha}$ on $M(A)$. Then $e_s=\overline{\alpha}(s)(1_{M(A)})$ is a strict limit of $\overline{\alpha}(s)(u_\lambda)=\alpha_s(u_\lambda)$ and thus it is the projection onto $E_{v_s}=\alpha_s(A)$. Moreover, $\delta$ is a unit in $M(E_{v_s})$ and the map $ m\to e_s m e_s$ implements an embedding $M(E_{v_s})\hookrightarrow M(A)$. On the other hand, one can  easily verify that $\overline{\alpha}(s)(M(A))\subset M(E_{v_s})$. Hence, $\overline{\alpha}(s)(M(A))= M(E_{v_s})$ and we also have that $\overline{\alpha}(s)(M(A))$ is a closed two-sided *-ideal in $M(A)$. By \autoref{lma:acta} we also obtain an action $\overline{\beta}$ of $S^*$ on $M(A)$, which clearly extends the action $\beta$ on $A$. 

By definition and the fact that non-degenerate representations of $A$ are extendible, the crossed product $A\rtimes_{\alpha}S$ is a closed ideal in $M(A)\rtimes_{\overline{\alpha}} S$ generated as a C*-algebra by $\{i_{M(A)}(a)i_S(s)|\ a\in A,\ s\in S\}$. Take $a,b\in A$, $s,t\in S$ and calculate the following products inside $M(A)\rtimes_{\overline{\alpha}} S$.
$$i_S(s)^*i_{M(A)}(a)=i_S(s)^*i_S(s)i_S(s)^*i_{M(A)}(a)= i_S(s)^*i_{M(A)}(e_sa)= $$ $$i_{M(A)}(\overline{\alpha}(s)^*(e_sa))i_S(s)^*,$$
$$i_{M(A)}(a)i_S(s)i_{M(A)}(b)i_S(t)=i_{M(A)}(a)i_S(s)i_{M(A)}(b)i_S(s)^*i_S(s)i_S(t)=$$ $$i_{M(A)}(a\alpha_s(b))i_S(s)i_S(t).$$
It follows that $A\rtimes_{\alpha}S$ is the closure of the linear span  of the set $\{i_{M(A)}(a)x|\ a\in A,\ x=i_S(s_1)^*i_S(s_2)...i_S(s_n),\ s_i\in S, \ n\in\mathbb{N}\}$. As usual, we call elements of the form $i_S(s_1)^*i_S(s_2)^*...i_S(s_n)^*$ monomials. Repeating the same reasoning as in \autoref{thm:tcross}, for any product $i_{M(A)}(a)i_S(s_1)^*i_S(s_2)...i_S(s_n)$ we may assume that $a\in E_{s_1^*s_2...s_n}$.
 
In the same way the crossed product $A\rtimes_{\beta} S^*$ can be viewed as a closed ideal in the unital crossed product $M(A)\rtimes_{\beta} S^*$. Denote by $j_{A,x}$ the embedding of $E_x$ 
 in $M(A)\rtimes_{\beta} S^*$. Then $A\rtimes_{\beta} S^*$ is the closure of the linear span of the set $\{j_{A,x}(a)\delta_x|\ x\in S^*\}$ with the algebraic structure given by Definition \ref{crossinv}. 

\autoref{thm:tcross} gives an isomorphism $\phi$ between the C*-algebras $M(A)\rtimes_{\overline{\alpha}} S$ and $M(A)\rtimes_{\overline{\beta}}S^*$. For any $a\in A$ and a monomial $x$ the map $\phi$ sends $i_{M(A)}(a)x$ to $j_{A,x}(a)\delta_x$. By the Lemma \ref{covrep}, the covariant representations of the systems $(\alpha,S,A)$ and $(\beta,S^*,A)$ are the same. Therefore,  $i_A(a)=0$ iff $j_{A,1}(a)=0$. Hence $\phi$ restricts to an isomorphism between $A\rtimes_\alpha S$ and $A\rtimes_{\beta} S^*$. 
  \end{proof}

\begin{cor} Let $\alpha$ be an action of $S$ on a locally compact Hausdorff space $X$ by continuous injective maps and $\beta$ the induced action of $S^*$. Then the crossed product C*-algebras $C_0(X)\rtimes_\alpha S$ and $C_0(X)\rtimes_\beta S^*$ are isomorphic. 
\end{cor}
\begin{proof} The action $\alpha$ induces an action $\tilde{\alpha}$ on $C_0(X)$ by the formula
$$\tilde{\alpha}_s(f)(x)=f(\alpha^*_s(x)),$$ 
where $f\in C_0(X)$, $x\in D_s$ in the notation of \autoref{lma:actsp}. Then clearly $E_s=C_0(D_s)\subset C_0(X)$, so $\tilde{\alpha}_s\colon C_0(X)\to C_0(D_s)$ and every $\tilde{\alpha}_s$ is injective. The action of $S^*$ given by $\tilde{\alpha}$ by \autoref{lma:acta} coincides with the action $\tilde{\beta}$ induced by  $\beta$ on $C_0(X)$. Therefore by Theorem \ref{crossall} the crossed products are isomorphic.
\end{proof}

\section{Partial crossed products, Ore semigroups}\label{secore}

The results of \cite{Milan} show that a C*-algebra of an E-unitary inverse semigroup $P$ is isomorphic to a partial crossed product  of the commutative subalgebra generated by idempotents in $P$ by a partial action of $G$, where $G$ is the maximal group homomorphic image of $P$.

\begin{dfn}
\emph{A partial action} $\alpha$ of a group $G$ on a set  $X$ is a pair $(\{D_g\}_{g\in G},\{\alpha_g\}_{g\in G})$, where $D_g$ are subsets of $X$ and $\alpha_g\colon D_{g^{-1}}\to D_g$  are bijections, satisfying for any $g,h\in G$:
\begin{enumerate}
	\item $D_1=X$, $\alpha_1=id_X$,
	\item $\alpha_g(D_{g^{-1}}\cap D_h)=D_g\cap D_{gh}$,
	\item $(\alpha_g\alpha_h)(x)=\alpha_{gh}(x)$ for all $x\in D_{h^{-1}}\cap D_{h^{-1}g^{-1}}$.
\end{enumerate}

\end{dfn}

\begin{thm}\label{thm:milan} (D.~Milan, B.~Steinberg \cite{Milan}.) Let $P$ be an E-unitary inverse semigroup with idempotent set $E$ and maximal group image $G$. Then there exists a partial action of $G$ on $E$ such that 
 $$C^*(P) \cong C^*(E)\rtimes G, \ C^*_r(P) \cong C^*(E)\rtimes_r G.$$
\end{thm}

In a view of \autoref{cor:eunitaryall}, we get an immediate Corollary.

\begin{cor}
If a semigroup $S$ is embeddable in a group, $C^*(S^F)\cong C^*(E^F)\rtimes G$, $C^*(S^*)\cong C^*(E)\rtimes G$, where $E^F$ and $E$ are subsemigroups of idempotents in $S^F$ and $S^*$ correspondingly, and $G$ is the group generated by $S$. The same holds for the reduced C*-algebras and reduced crossed products.
\end{cor}

The models for $S^F$ and $S^*$ described in \autoref{thm:modelsf} and \autoref{cor:modelss} give us concrete formulas for the partial actions in Corollary above.

For every $g\in G$ define using the notation of \autoref{thm:modelsf}:
$$D_{g^{-1}}=\{[A]\in E^F\colon \mbox{ there exist } g_1,g_2\in A \mbox{ such that } g_1\leq g^{-1}\leq g_2\},$$ 
or equivalently $D_{g^{-1}}=\{[A]\colon [\{1,g^{-1}\}]\leq [A]\}$. For $[A]\in D_{g^{-1}}$ define
$$\alpha_g([A])=g[A]=[\{1\}\cup\{gh\colon h\in A\}]$$
One can easily see that $\alpha_g$ is a bijection between $D_{g^{-1}}$ and $D_{g^{-1}}$ and the formula 3) from the Definition is satisfied as well. Hence, $\alpha$ is a partial action of $G$ on $E^F$, and it coincides with the partial action in \autoref{thm:milan}. 

Consider $S^*$ and its idempotent subsemigroup $E$. Now $[A]$ denotes an equivalence class defined in \autoref{cor:modelss}. For every $g\in G$ we define
$$\tilde{D}_{g^{-1}}=\{[A]\in E\colon g\in A\cdot S^{-1} \},$$ 
 where $A\cdot S^{-1}$ denotes pointwise product of sets and $S^{-1}=\{a^{-1}\colon a\in S\}\subset G$. For $[A]\in \tilde{D}_{g^{-1}}$ define
$$\tilde{\alpha}_g([A])=g[A]=[\{1\}\cup\{gh\colon h\in A\}]$$
Again, $\tilde{\alpha}$ is a partial action of $G$ on $S^*$, which gives an isomorphism $C^*(S^*)\cong C^*(E)\rtimes G$.

\begin{pgr}\label{pgr:ore}

For a particular class of semigroups we can say more about the connection between $S^*$ and partial crossed products of groups.

\begin{dfn}
A partial action $\alpha$ of a group $G$ on a C*-algebra $A$ is a pair $(\{E_g\}_{g\in G},\{\alpha_g\}_{g\in G})$, where $E_g$ are closed two-sided *-ideals in $A$ and $\alpha_g\colon E_{g^{-1}}\to E_g$  are *-isomorphisms, satisfying for any $g,h\in G$:
\begin{enumerate}
	\item $E_1=A$,
	\item $\alpha_g(E_{g^{-1}}\cap E_h)=E_g\cap E_{gh}$,
	\item $(\alpha_g\alpha_h)(x)=\alpha_{gh}(x)$ for all $x\in E_{h^{-1}}\cap E_{h^{-1}g^{-1}}$.
\end{enumerate}
Then $(A,G,\alpha)$ is a C*-partial dynamical system.
\end{dfn}

In \cite{Exel2} it was shown, that the partial actions and partial representations of a group $G$ are in one-to-one correspondence with actions and representations of a special inverse semigroup $S(G)$. Moreover, an isomorphism between a partial crossed product by $G$ and a crossed product by $S(G)$ was proved in \cite{Exel2} and earlier in \cite{Sieben}. We recall these results.

Following \cite{Exel2} $S(G)$ is a semigroup generated by elements $t_g$ for all $g\in G$ satisfying the following relations:

	\begin{equation}\label{tg1} t_{g^{-1}}t_gt_h=t_{g^{-1}}t_{gh}\end{equation}
	\begin{equation}\label{tg2} t_gt_ht_{h^{-1}}=t_{gh}t_{h^{-1}}\end{equation}
	 \begin{equation}\label{tg} t_gt_1=t_g \end{equation}

Then $S(G)$ is an inverse semigroup with unit $t_1$ and involution $t_g^*=t_{g^{-1}}$. A partial representation of $G$ on a Hilbert space $H$ is a map $T\colon G\to B(H)$, sending $g\to T_g$, where $T_g$ satisfy relations (\ref{tg1})--(\ref{tg}).

\begin{lma}\label{ex1}(\cite{Exel2}) Partial actions (partial representations) of $G$ are in one-to-one correspondence with actions (resp. *-representations) of $S(G)$.   \end{lma}

\begin{thm}\label{ex2} (\cite{Exel2}). Let $\alpha\colon G\to I(A)$ be a partial action of a group $G$ on a C*-algebra $A$, and $\beta$ the action of $S(G)$ induced by $\alpha$. Then  $A\rtimes_\alpha G$ and $A\rtimes_\beta S(G)$ are isomorphic.
\end{thm}

 Let us consider a particular case of a group. Namely, let $S$ be a left Ore semigroup, so that by  Theorem \ref{ore} of Ore and Dubreil there exists a group $G$ such that $G=S^{-1}S$. Now we study the connections between  $S(G)$ and the inverse semigroup $S^*$ generated by $S$ as defined in Section \ref{s1}.

In $S(G)$ all elements are partial isometries, including $t_s$ for all $s\in S$. Moreover, $S(G)$ is generated not only by elements corresponding to $S$. It follows that $S(G)$ is not isomorphic to $S^F$ or to $S^*$. We implement the choice of generators by setting which elements should be represented by isometries. Namely, define on $S(G)$ a relation $t_s^*t_s\sim 1$ for all $s\in S$ and denote by $R$ the generated congruence.

\begin{lma}\label{sgcong} The quotient semigroup of $S(G)$ by the congruence $R$ is isomorphic to $S^*$. \end{lma}
\begin{proof}
For all $g\in G$ we denote by $\tilde{t}_g$ the image of $t_g$ under the quotient map $S(G)\to S(G)/R$. Then $S(G)$ is characterised by $R$ and the conditions (\ref{tg1})--(\ref{tg}). For any $s,p\in S$ set $g=s^{-1}$, $h=p$ and consider equation (\ref{tg1}):
$$\tilde{t}_s\tilde{t}_s^*\tilde{t}_p=\tilde{t}_s\tilde{t}_{s^{-1}p}$$
Multiplying from the left by   $\tilde{t}_s^*$ we obtain $\tilde{t}_{s^{-1}p}=\tilde{t}_s^*\tilde{t}_p$. In the same way setting $g=s$, $h=p$ in (\ref{tg2}) we get $\tilde{t}_{sp}=\tilde{t}_{s}\tilde{t}_{p}$. Since  a quotient of an inverse semigroup is an inverse semigroup, we deduce that $S(G)/R$ satisfies the definition of $S^*$ having $\tilde{t}_{s}$ at the place of $v_s$. The converse is also true, let us show for instance (\ref{tg1}) for the generators of $S^*$. Take arbitrary $g=p^{-1}q$, $h=s^{-1}r$ where $p,q,r,s\in S$. Let $qs^{-1}=a^{-1}b$ for some $a,b\in S$, then $gh=p^{-1}qs^{-1}r=p^{-1}a^{-1}br$ and we have
$$ v_av_q=v_bv_s \Longrightarrow v_s^*=v_q^*v_a^*v_b$$
Then the left hand side of (\ref{tg1}) equals 
$$(v_q^*v_p)(v_p^*v_q)(v_s^*v_r)=v_q^*v_pv_p^*v_qv_q^*v_a^*v_bv_r=$$ $$=v_q^*v_qv_q^*v_pv_p^*v_a^*v_bv_r=(v_q^*v_p)(v_p^*v_a^*v_bv_r)$$
Therefore, the map $\tilde{t}_{s^{-1}p}\to v_s^*v_p$ is an isomorphism between $S(G)/R$ and $S^*$.\end{proof}

Combining this result with Proposition 6.3 and Theorem 4.2 (presented here as Lemma \ref{ex1})  in \cite{Exel} and Lemmas \ref{reprs}, \ref{lma:actsp}, \ref{lma:acta}, we immediately get the following.

\begin{cor}\label{par}  Any isometric inverse representation of $S$  induces a partial representation of $G$.  Any injective action of $S$ on a space $X$ induces a partial action of $G$ on $X$. Any injective action of $S$ on a C*-algebra $A$ with ideal images induces a partial action of $G$ on $A$. \end{cor}
\begin{rmk}\label{remsg} The reverse statement is true under some conditions. If a *-representation $T$ of $S(G)$ on some Hilbert space $H$ factors through the quotient map $S(G)\to S(G)/R$, then clearly $T$ induces a *-representation of $S^*$, which we denote by the same symbol. It follows that any partial representation $T$ of $G$ which satisfies the property that $T_s$ are isometries for all $s\in S$, gives a *-representation of $S^*$ (and an isometric inverse representation of $S$). \end{rmk}

\begin{thm} Let $S$ be a left Ore semigroup, and let $\alpha$ be an  extendible injective action of $S$ on a C*-algebra $A$, and $\tilde{\alpha}$ the induced partial action of $G$ on $A$. Then the crossed product $A\rtimes_\alpha S$ is isomorphic to the partial crossed product $A\rtimes_{\tilde{\alpha}} G$. 
\end{thm}
\begin{proof} By Theorems \ref{thm:tcross} and \ref{crossall}, the crossed product $A\rtimes_\alpha S$ is isomorphic to the inverse semigroup crossed product $A\rtimes_\alpha S^*$.   On the other hand, by the Lemma \ref{ex1} the partial action $\tilde{\alpha}$ induces an action $\beta$ of $S(G)$ on $A$, and by Theorem \ref{ex2} the crossed products $A\rtimes_\alpha G$ and $A\rtimes_\beta S(G)$ are isomorphic. So, it remains to prove the isomorphism between $A\rtimes_\alpha S^*$ and $A\rtimes_\beta S(G)$.

Let $(\pi,T)$ be a covariant representation of $(\beta,S(G),A)$. We have $E_{t_s^*}=E_{\tilde{t}_s^*}=A$ due to the fact that $\beta$ is induced by $\alpha$. By condition (2) of Definition \ref{covinv}, for any $s\in S$ the operator $T_{t_s}^*T_{t_s}=T_{t_s^*t_s}$ is a partial isometry with initial and final spaces equal to $H$, hence a unitary; and at the same time it is a projection. It implies that $T_{t_s}$ is an isometry on $H$. By Remark \ref{remsg}, the representation $T\colon S(G)\to B(H)$ factors through the quotient map $S(G)\to S^*$ defined in Lemma \ref{sgcong} and gives a representation $\tilde{T}$ of $S^*$ on $H$. So, we obtain a covariant representation $(\pi,\tilde{T})$ of $(\alpha,S^*,A)$. 

If $(\pi,T)$ is a covariant representation of  $(\alpha,S^*,A)$, then  Corollary \ref{par} gives a covariant representation of $(\beta,S(G),A)$. Therefore these dynamical systems have the same set of covariant representations. The underlying algebra $L$ for the two crossed products are different, but the completions under the supremum norm over all covariant representations are isomorphic. Indeed, if $s,t\in S(G)$ are such that $s\sim_R t$, then for any $a\in E_s=E_t$ and for any covariant representation $(\pi,T)$ we have $$\pi\times T(a\delta_s)=\pi(a)T_s=\pi(a)T_t=\pi\times T(a\delta_t).$$
Thus, $A\rtimes_\alpha S^*$ and $A\rtimes_\beta S(G)$ are isomorphic.
\end{proof}

\end{pgr}
\section{Conclusion}

%The construction of a semigroup C*-algebra is different from the one for group C*-algebra. In the group case the C*-algebra is a closed linear span of operators which represent group elements. In the semigroup case we take generators corresponding to semigroup elements, then (in need of involution) add their conjugates and generate by this set a semigroup. The latter turns out to be an inverse semigroup. The C*-algebra is then a closed linear span of this inverse semigroup.

One can see that the semigroup C*-algebras (both universal and reduced) of a cancellative semigroup $S$  are in fact C*-algebras (the full C*-algebra and a C*-algebra generated by some special representation) of some inverse semigroup. All phenomena of these algebras, discussed in the Introduction, can be explained by this fact, and indicate that the concept of these algebras is imperfect. We have shown that to every left cancellative semigroup $S$ one can associate a universal inverse semigroup $S^*$. Then the full and reduced C*-algebras of $S^*$ do not have the mentioned problems and can be regarded as ``new'' C*-algebras of $S$. The universal inverse semigroup  captures many properties of $S$, of its ``old'' C*-algebras and also of actions and crossed products by $S$. All together this convinces us that $S^*$  serves the purpose of describing the C*-theory of $S$.  

\vspace{3mm}

\textbf{Acknowledgements}. The author is thankful to Prof. S.~Echterhoff for many  useful discussions and to Dr.~H.~Thiel for pointing at mistakes. The research was supported by the Alexander von Humboldt Foundation. A part of this research was supported by the Russian Foundation of Basic Research (RFBR), grant 14-01-31358.

\end{document}